\documentclass[a4paper,11pt]{amsart}
\usepackage{amsmath,amsfonts,amsthm,amssymb,enumerate,ifpdf}
\usepackage[pdftex]{graphicx}

\numberwithin{equation}{section}

\newtheorem{theorem}{Theorem}[section]

\newtheorem{proposition}[theorem]{Proposition}

\newtheorem{claim}[theorem]{Claim}
\newtheorem{lemma}[theorem]{Lemma}

{
\theoremstyle{definition}
\newtheorem{definition}[theorem]{Definition}

\newtheorem{remark}[theorem]{Remark}
}

\title{On addition of 1-handles with chart loops to 2-dimensional braids}
\author{Inasa Nakamura}
\address{
Institute for Biology and Mathematics of Dynamic Cellular Processes (iBMath), Interdisciplinary Center for Mathematical Sciences, Graduate School of Mathematical Sciences, The University of Tokyo\newline
3-8-1 Komaba, Tokyo 153-8914, Japan}
\email{inasa@ms.u-tokyo.ac.jp}

\begin{document}

\begin{abstract}
A 2-dimensional braid over an oriented surface-knot $F$ is presented by a graph called a chart on a surface diagram of $F$. 
We consider 2-dimensional braids obtained by an addition of 1-handles equipped with chart loops. We introduce moves of 1-handles with chart loops, called 1-handle moves, and we investigate how much we can simplify a 2-dimensional braid by using 1-handle moves. Further, we show that an addition of 1-handles with chart loops is an unbraiding operation. 
\end{abstract}

\maketitle

\section{Introduction}\label{sec1}

A {\it surface-knot} is the image of a smooth embedding of a connected closed surface into the Euclidean 4-space $\mathbb{R}^4$. 
In this paper, we assume that surface-knots are oriented. 
For a surface-knot $F$, we can consider a surface in the form of a covering over $F$, called a 2-dimensional braid over $F$. Two 2-dimensional braids over $F$ are equivalent if one is carried to the other by an ambient isotopy of $\mathbb{R}^4$ whose restriction to a tubular neighborhood of $F$ is fiber-preserving. 
A 2-dimensional braid over $F$, denoted by $(F, \Gamma)$, is presented by a graph $\Gamma$ called a chart on a surface diagram of $F$. For simplicity, we will often 
identify a surface diagram of $F$ with $F$ itself. 
In \cite{Hirose}, Hirose investigated particular 2-dimensional braids over a connected surface $\Sigma$ standardly embedded in $\mathbb{R}^4$, called toroidal knotted surfaces, and he showed that any such surface is classified into two types, the connected sum of trivial tori with the spun $T^2$-knot of a classical knot and that of those with the turned spun $T^2$-knot, by using the generators of the group of isotopies of $\Sigma$ which are extendable to $\mathbb{R}^4$ such that it is a subgroup of the mapping class group of $\Sigma$. This result immediately implies the same result for any 2-dimensional braids with \lq\lq repeated pattern" over $\Sigma$, which is presented by a chart consisting of loops on $\Sigma$ satisfying a certain condition. In particular, this result implies that all the chart loops presenting the 2-dimensional braid can be gathered to a torus part of $\Sigma$. 
Our first motivation of this paper is to give a graphical proof of this result. 
We introduce equivalence moves of surface diagrams with charts, called 1-handle moves, and 
we investigate how much we can simplify the 2-dimensional braid by using 1-handle moves. 

Let $B^2$ be a unit 2-disk and let $I=[0,1]$. 
A {\it 1-handle} is a 3-ball $h=B^2 \times I$ smoothly embedded in $\mathbb{R}^4$ such that $h \cap F=(B^2 \times \partial I) \cap F$. Further we assume that $h$ has the framing such that the projected image in $\mathbb{R}^3$ has the blackboard framing (see Remark \ref{rem4-1}). 
The surface-knot obtained from $F$ by a {\it 1-handle addition} along $h$ is the surface

\[
(F-(\mathrm{Int} B^2 \times \partial I)) \cup (\partial B^2 \times I),
\]
which is denoted by $F+h$. In this paper, we assume that $h$ is orientable, that is, $F+h$ is orientable, and we give $F+h$ the orientation induced from that of $F$. 

In the first part of this paper, we consider $F$ as a surface-knot which is in the form of the result of 1-handle additions for a surface-knot $F_0$. For a 1-handle $h$, we call the {\it oriented core} the image of $\{0\} \times I \subset h$ with the orientation of $I$. Take a base point $x$ in $\partial N(F_0)$ for a tubular neighborhood $N(F_0)$ of $F_0$ in $\mathbb{R}^4$. For the oriented core $C$, we denote the closure of $C-C \cap N(F_0)$ by $\overline{C}$. 
 Take a path $\alpha$ (respectively, $\beta$) in $\partial N(F_0)$ connecting $x$ and the initial point (respectively, the terminal point) of $\overline{C}$. Then the closed path $\alpha \overline{C} \beta^{-1}$ induces an element in the double coset $P \backslash G(F_0)/P$, where  $G(F_0)=\pi_1(\mathbb{R}^4-\mathrm{Int}N(F_0), x)$, the knot group of $F_0$,  and $P$ is the image $i_*(\pi_1(\partial N(F_0), x))$ by the homomorphism $i_*$ induced by the inclusion $i: \partial N(F_0) \hookrightarrow \mathbb{R}^4-F_0$, the peripheral subgroup of $G(F_0)$. 
Since a 1-handle is determined by its oriented core, and two oriented cores $C$ and $C'$ are \lq\lq equivalent" if and only if $P(C)P=P(C')P$ \cite{Boyle1} (see also \cite{Kamada14}), so 
we identify a 1-handle $h$ with an element in $P \backslash G(F_0) / P$. 

For a 1-handle $h$ with the oriented core $C$, we determine the {\it core loop} of $h$ by the projected image of $\alpha \overline{C} \beta$ to $F_0+h \subset F$, with the orientation induced from that of $C$, where $\alpha$, $\beta$, $\overline{C}$ are given for $C$ as in the above paragraph. For a set of 1-handles, we add a condition that core loops are mutually disjoint. 
We determine the {\it cocore} of a 1-handle $h$ by the oriented closed path $\partial B^2 \times \{0\} \subset h$, with the orientation of $\partial B^2$. 
Further, we determine the base point of the core loop and the cocore of $h$ by their intersection point (see Figure \ref{fig1}). 

\begin{figure}[ht]
\centering
\includegraphics*[height=3.5cm]{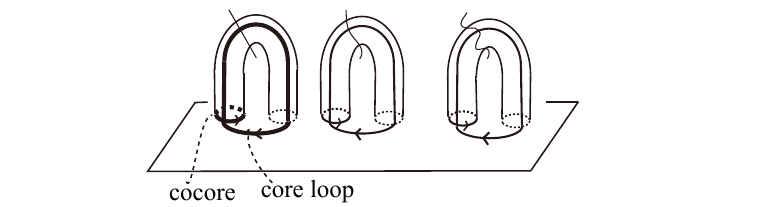}
\caption{The core loop and cocore of a 1-handle. In order to indicate that a 1-handle may be non-trivial, we draw a slash in the middle.}
\label{fig1}
\end{figure}
 
A 2-dimensional braid {\it with repeated pattern} is a 2-dimensional braid presented by a chart consisting of a finite number of bands of parallel loops such that all bands of parallel loops present the same classical braid $b$, which is called the {\it pattern braid}. 
Such a 2-dimensional braid is determined from the integers which present the numbers of the bands intersecting the oriented closed curves presenting the generators of the first homology group $H_1(F; \mathbb{Z})$. 
When $F$ is an embedding of a closed surface $\Sigma$ of genus $g$, we present the generators of $H_1(F; \mathbb{Z})$ by the embeddding of those of $H_1(\Sigma; \mathbb{Z})$  determined by the curves $\alpha_1, \alpha_2, \ldots, \alpha_{2g}$ as illustrated in Figure \ref{fig1-2}, and in particular, 
we take such curves as the cocores and core loops of 1-handles. 
For integers $m$ and $n$, let us denote by $h(m,n)$ a 1-handle $h \in P \backslash G(F_0) / P$ with a chart such that the cocore and core loop 
presents the pattern braid to the power $m$ and $n$, respectively. This presentation $h(m,n)$ is well-defined; hence  we can assume that $h(m,n)$ is presented by \lq\lq simplified" chart loops on regular neighborhoods of the core loop and cocore of $h$ in $F$, which goes along the core loop and the cocore $m$ and $n$ times, respectively; we call such a 1-handle a {\it 1-handle with chart loops}, or simply a {\it 1-handle}; further, we can assume that all 1-handles are attached to a fixed 2-disk as illustrated in Figure \ref{fig1}, and there are no chart edges on the 2-disk except those belonging to the 1-handles (see Section \ref{sec3}). 

\begin{figure}[ht]
\centering
\includegraphics*[height=3.5cm]{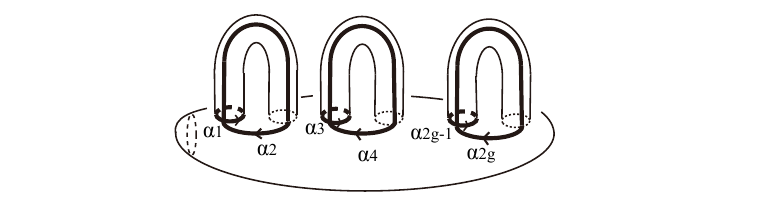}
\caption{Generators of $H_1(\Sigma; \mathbb{Z})$ for a closed surface $\Sigma$ of genus $g$.}
\label{fig1-2}
\end{figure}

For a 2-dimensional braid $(F_0, \Gamma)$ and 1-handles with chart loops $h_1(m_1,n_1)$, $\ldots$, $h_g(m_g, n_g)$, we denote the 2-dimensional braid which is the result of a 1-handle addition by $(F_0, \Gamma) + \sum_{j=1}^g h_j(m_j,n_j)$, which presents a 2-dimensional braid over $F_0+\sum_{j=1}^g h_j$ with repeated pattern. In particular, when $F_0$ is a knotted 2-sphere and $\Gamma=\emptyset$, an empty chart, for simplicity, we denote the resulting surface by $\sum_{j=1}^g h_j(m_j, n_j)$. 

Using this notation, Hirose's result is presented as follows. 
A {\it trivial} 1-handle is a 1-handle whose oriented core is represented by $1 \in  P\backslash G(F_0)/P$. 

\begin{theorem}[Hirose]\label{thm1}

A 2-dimensional braid with repeated pattern over a standard surface, $\sum_{j=1}^g 1(m_j, n_j)$, is equivalent to one of the followings:

\begin{enumerate}
\item[\rm{(1)}]
$1(k, 0)+\sum _{j=2}^g 1(0,0)$,

\item[\rm{(2)}]
$1(k, k)+\sum_{j=2}^g 1(0,0)$,
\end{enumerate}
for an integer $k$.

\end{theorem}

Our results are as follows. 

\begin{theorem}\label{thm2}
Let $(F, \Gamma)=\sum_{j=1}^g h_j(m_j, n_j)$ be a 2-dimensional braid with repeated pattern over the surface-knot $F$ obtained from a surface-knot $F_0$ by an addition of 1-handles. Then, by 1-handle moves (see Section \ref{sec4}), $(F, \Gamma)$ is deformed to the following form:

\begin{equation*}
 h_1' (m_, n_1')+ h_2'(0, n_2') +\sum_{j=3}^g h_j'(0,0),
\end{equation*}
where $m=\mathrm{gcd}_{1\leq j\leq g}\{m_j\}$, the greatest common divisor of $m_1, \ldots, m_g$, 
$h_1',\ldots, h_g' \in \langle h_1, \ldots, h_g \rangle < P\backslash G(F_0)/P$, where $\langle h_1, \ldots, h_g \rangle$ is the subgroup of $P\backslash G(F_0)/P$ generated by $ h_1, \ldots, h_g$, and $n_1', n_2'$ are integers.
\end{theorem}

In particular, we can prove Theorem \ref{thm1}. 

\begin{theorem}\label{thm3}
Let $(F, \Gamma)$ be as in Theorem \ref{thm2}. 
Then, by 1-handle moves, $(F, \Gamma) + 1(0,0)$ is deformed to the following form:

\begin{equation*}
 \sum_{j=1}^g h_j(0, \tilde{n}_j)+ h(m,n),
\end{equation*}
where $h \in \langle h_1, \ldots, h_g \rangle  < (P\backslash G(F_0)/P)/H$ for the normal subgroup $H$ of $P\backslash G(F_0)/P$ generated by $h_ih_jh_i^{-1}h_j^{-1}$ $(i,j=1, \ldots, g)$,  $m=\mathrm{gcd}_{1\leq j\leq g}\{m_j\}$, $n= \frac{\sum_{j=1}^g m_j n_j}{m}$ and $\tilde{n}_1, \ldots, \tilde{n}_g \in \{0,\ldots, m-1\}$. 
\end{theorem}

We improve Theorem \ref{thm1}. 

\begin{theorem}\label{thm4}
Let $\mathrm{gcd}=\mathrm{gcd}_{1\leq j\leq g} \{m_j, n_j\}$. Then, a 2-dimensional braid with repeated pattern, $\sum_{j=1}^g 1(m_j, n_j)$, is equivalent to one of the followings:

\begin{enumerate}
\item[\rm{(1)}] $1(\mathrm{gcd}, \mathrm{gcd})+ \sum_{j=2}^g 1(0,0)$ if $\frac{\sum_{j=1}^g m_jn_j}{(\mathrm{gcd})^2}$ is odd, \\
\item[\rm{(2)}] $1(\mathrm{gcd}, 0)+\sum_{j=2}^g 1(0,0)$ if $\frac{\sum_{j=1}^g m_jn_j}{(\mathrm{gcd})^2}$ is even. \end{enumerate}
\end{theorem}

We consider $(F, \Gamma)$ with repeated pattern for any surface-knot $F$. 
By an addition of a 1-handle $H=1(1,0)$ to $(F, \Gamma)$, we can gather all the chart loops on $H$. 

\begin{theorem}\label{thm5}
 For $(F, \Gamma)$ with repeated pattern for any surface-knot $F$, by 1-handle moves, $(F, \Gamma)+1(1,0)$ is deformed to 
\begin{equation*}
(F, \emptyset) + h(1,\delta),
\end{equation*}
where $h$ is a 1-handle attached to $F$ and $\delta=0$ or $1$. 
\end{theorem}

Next we consider $(F, \Gamma)$ for any surface-knot $F$ where we may remove the condition of connectedness, and a chart $\Gamma$ of degree $N$ which does not contain black vertices (see Section \ref{sec2-2}): we consider a 2-dimensional braid over $F$ without branch points. 
We denote by $h(a,b)$ a 1-handle $h$ with a chart without black vertices, attached to $F$, such that the chart is contained in the union of regular neighborhoods in $F+h$ of the core loop and cocore, and the cocore and the core loop present braids $a$ and $b$, respectively. In particular, we consider 1-handles with chart loops. Let $\sigma_1, \ldots, \sigma_{N-1}$ be the standard generators of the braid group $B_N$. We denote by $h(e, e)$ the 1-handle $h$ with an empty chart, by $h(\sigma_i, e)$ $h$ with the chart consisting of a loop with the label $i$ along the core loop, and by $h(\sigma_i, \sigma_j^{\epsilon})$ $(|i-j|>1, \epsilon\in \{+1, -1\})$ $h$ with the chart consisting of a chart loop along the core loop with the label $i$ at the base point (see Section \ref{sec2-2}) and a loop along the cocore with the label $j$, with orientations determined from the signs of $\sigma_i$ and $\sigma_j^\epsilon$. 
Note that $h(e,e)$ equals $h(0,0)$ used for the repeated pattern. 
Different from 2-dimensional braids with repeated pattern, a 1-handle with chart loops is not always determined only from its presentation (Remark \ref{rem6-1}); hence, except in special cases such as when $\Gamma$ is an empty chart,  
we need to assign where we add a 1-handle. 
 
\begin{theorem}\label{thm6}
Let $(F, \Gamma)$ be a 2-dimensional braid for any surface-knot $F$ and a chart $\Gamma$ without black vertices. 
Let $N$ be the degree of $\Gamma$. 
By an addition of finitely many 1-handles in the form $1(\sigma_i, e)$ or $1(e,e)$ $(i \in \{1, \ldots,N-1\})$, to appropriate places in $F$, $(F, \Gamma)$ is deformed to 
\begin{equation}\label{eq:6-2}
(F, \emptyset) +\sum_\lambda H_\lambda,
\end{equation}
where $H_\lambda=h_\lambda(\sigma_i, e)$,  $h_\lambda(\sigma_i, \sigma_j^{\epsilon})$ or $h_\lambda(e,e)$ for 1-handles $h_\lambda$ attached to $F$, $i,j \in \{1.\ldots,N-1\}, |i-j|>1$ and $\epsilon \in \{+1, -1\}$. 

In particular, by an addition of 1-handles $\sum_{j=1}^{N-1}1(\sigma_j, e)$ and finitely many $1(e,e)$, to a fixed 2-disk in $F$, $(F, \Gamma)$ is deformed to 
\begin{equation}\label{eq12}
(F, \emptyset) +\sum_{j=1}^{N-1}h_j(\sigma_j, e)+\sum_\lambda 1_\lambda,
\end{equation}
where $h_1, \ldots, h_{N-1}$ are 1-handles attached to $F$, and $1_\lambda=1(\sigma_i, \sigma_j^{\epsilon})$ or $1(e,e)$ $(i,j \in \{1.\ldots,N-1\}, |i-j|>1$ and $\epsilon \in \{+1, -1\})$. 
\end{theorem}

Note that since the resulting chart on $F$ is an empty chart, the presentations (\ref{eq:6-2}) and (\ref{eq12}) are well-defined. 

\begin{definition}\label{def1-7}
We call the minimal number of 1-handles necessary to make $(F, \Gamma)$ in the form (\ref{eq:6-2}) the {\it weak unbraiding number} of $(F, \Gamma)$, which is denoted by $u_w(F, \Gamma)$. 
\end{definition}

\begin{theorem}\label{thm7}
Let $(F, \Gamma)$ be a 2-dimensional braid as in Theorem \ref{thm6}. 
 By an addition of finitely many 1-handles in the form $1(\sigma_i, e)$, $1(\sigma_i, \sigma_j^\epsilon)$ or $1(e,e)$ $(i, j \in \{1, \ldots,N-1\}, |i-j|>1, \epsilon \in \{+1, -1\})$, to appropriate places in $F$, a 2-dimensional braid $(F, \Gamma)$ is deformed to  
\begin{equation}\label{eq:7}
(F, \emptyset) +\sum_{\lambda} H_\lambda,
\end{equation}
where $H_\lambda=h_\lambda(\sigma_i, e)$ or $h_\lambda(e,e)$ 
for 1-handles $h_\lambda$ attached to $F$ and  $i \in \{1,\ldots, N-1\}$.

In particular, by an addition of 1-handles $\sum_{j=1}^{N-1}1(\sigma_j, e)$ and finitely many 1-handles in the form $1(\sigma_i, \sigma_j^\epsilon)$ $(i, j \in \{1, \ldots,N-1\}, |i-j|>1, \epsilon \in \{+1, -1\})$ or $1(e,e)$, to a fixed 2-disk in $F$, $(F, \Gamma)$ is deformed to 
\begin{equation*}
(F, \emptyset) +\sum_{j=1}^{N-1}h_j(\sigma_j, e)+\sum_\lambda 1(e,e),  
\end{equation*}
where $h_1, \ldots, h_{N-1}$ are 1-handles attached to $F$. 
\end{theorem}

\begin{definition}\label{def1-9}
We call the minimal number of 1-handles necessary to make $(F, \Gamma)$ in the form (\ref{eq:7}) the {\it unbraiding number} of $(F, \Gamma)$, which is denoted by $u(F, \Gamma)$. 
\end{definition}

\begin{proposition}\label{prop1}
Let $(F, \Gamma)$ be a 2-dimensional braid for any surface-knot $F$ and a chart $\Gamma$ of degree $N$ without black vertices. 
Then we have 
\[
 u_w(F, \Gamma) \leq u(F, \Gamma)\leq u_w(F, \Gamma)+c_{\mathrm{alg}}(\Gamma), \]
where $c_{\mathrm{alg}}(\Gamma)$ is the sum of the absolute values of algebraic sums of the numbers of crossings of type $(i,j)$ in $\Gamma$ for $i<j$ $(i,j \in \{1,\ldots,N-1\})$ (see Definition \ref{def6-7}). 
\end{proposition}

A chart edge is called a {\it free edge} if it is connected with two black vertices at its end points. 
A chart consisting of free edges is called an {\it unknotted chart}, which, if drawn on the standard surface, presents an unknotted surface-knot \cite{Kamada99, Kamada02}.
It is known \cite{Kamada99} that an addition of free edges to a chart $\Gamma$, to appropriate places, deforms $\Gamma$ to an unknotted chart. 
The {\it unknotting number}, denoted by $u(\Gamma)$, of a chart $\Gamma$ is the minimal number of such free edges necessary to make $\Gamma$ an unknotted chart. For a chart $\Gamma$ of degree $N$, \cite{Kamada99, Kamada02} implies that $u(\Gamma) \leq w(\Gamma)+N-1$, where $w(\Gamma)$ is the number of white vertices.  

\begin{proposition}\label{prop2}
For a 2-dimensional braid $(F, \Gamma)$ as in Proposition \ref{prop1}, we have
\[
u_w(F, \Gamma) \leq w(\Gamma)+2c(\Gamma)+N-1, 
\]
where $w(\Gamma)$ and $c(\Gamma)$ are the numbers of white vertices and crossings in $\Gamma$, respectively. 
\end{proposition}

Further we consider $(F, \Gamma)$ for any surface-knot $F$ and any chart $\Gamma$. Then Theorems \ref{thm6} and \ref{thm7} hold true when we change the resulting $(F, \emptyset)$ to $(F, \Gamma_0)$, where $\Gamma_0$ is an unknotted chart. 
Propositions \ref{prop1} and \ref{prop2} hold true with unchanged (see Section \ref{sec7}). 

Let $b(\Gamma)$ be the number of black vertices in $\Gamma$. If $b(\Gamma) \geq 2(N-1)$, then we can simplify the results in Theorems \ref{thm6} and \ref{thm7} as follows. 

\begin{theorem}\label{thm12}
Let $(F, \Gamma)$ be a 2-dimensional braid for any surface-knot $F$ and any chart $\Gamma$. 
Let $N$ be the degree of $\Gamma$ and let $b(\Gamma)$ be the number of black vertices in $\Gamma$. If $b(\Gamma) \geq 2(N-1)$, then, 
 by an addition of 1-handles $\sum_{j=1}^{N-1}1(\sigma_j, e)$ and finitely many $1(e,e)$, to a fixed 2-disk in $F$, $(F, \Gamma)$ is deformed to 
\begin{equation*}
(F, \Gamma_0)+\sum_\lambda 1(e,e),
\end{equation*}
where $\Gamma_0$ is an unknotted chart. 
\end{theorem}

The paper is organized as follows. In Section \ref{sec2}, we review 2-dimensional braids and their chart presentation, and we review equivalence moves of charts: C-moves and Roseman moves. In Section \ref{sec3}, we give a precise definition of 2-dimensional braids with repeated pattern. In Section \ref{sec4}, we introduce 1-handle moves. In Section \ref{sec5}, we give proofs of Theorems \ref{thm2}--\ref{thm4}, and we also give an alternative proof of Theorem \ref{thm1}. In Section \ref{sec6}, we give proofs of Theorems \ref{thm5}--\ref{thm7} and Propositions \ref{prop1} and \ref{prop2}. In Section \ref{sec7}, we prove Theorem \ref{thm12}. In Section \ref{sec8}, we give an example. 

\section{Two-dimensional braids and their chart presentations}\label{sec2}
In this section, we review 2-dimensional braids over a surface-knot \cite{N4}, which is an extended notion of 2-dimensional braids or surface braids over a 2-disk \cite{Kamada92, Kamada02, Rudolph}. A 2-dimensional braid over a surface-knot $F$ is presented by a finite graph called a chart on a surface diagram of $F$ \cite{N4} (see also \cite{Kamada92, Kamada02}). For two 2-dimensional braids of the same degree, they are equivalent if their surface diagrams with charts are related by a finite sequence of ambient isotopies of $\mathbb{R}^3$, and local moves called C-moves \cite{Kamada92, Kamada02} and Roseman moves \cite{N4} (see also \cite{Roseman}). 
 
\subsection{Two-dimensional braids over a surface-knot}
 
Let $B^2$ be a 2-disk, and let $N$ be a positive integer. 
For a surface-knot $F$, let $N(F)=B^2 \times F$ be a tubular neighborhood of $F$ in $\mathbb{R}^4$. 

\begin{definition}
A closed surface $S$ embedded in $N(F)$ is called a {\it 2-dimensional braid over $F$ of degree $N$} if it satisfies the following. 

\begin{enumerate}[(1)]
\item
The restriction $p|_{S} \,:\, S \rightarrow F$ is a branched covering map of degree $N$, where $p\,:\, N(F) \to F$ is the natural projection with respect to a framing of $N(F)$. 

\item The number of points consisting $S \cap p^{-1}(x)$ is $N$ or $N-1$ for any point $x \in F$.
\end{enumerate}
Take a base point $x_0$ of $F$. 
Two 2-dimensional braids over $F$ of degree $N$ are {\it equivalent} if there is a fiber-preserving ambient isotopy of $N(F)=B^2 \times F$ rel $p^{-1}(x_0)$ which carries one to the other. 

We define the {\it standard} 2-dimensional braid over $F$ to be the 2-dimensional braid  presented by an empty chart on a surface diagram of $F$, defined in \cite{N4}, and we define the {\it standard framing} of $N(F)$ to be the framing determined from the standard 2-dimensional braid. 

\end{definition}
\subsection{Chart presentation of 2-dimensional braids}\label{sec2-2}

Let $S$ be a 2-dimensional braid over a surface-knot $F$. A {\it surface diagram} of a surface-knot is the image of $F$ in $\mathbb{R}^3$ by a generic projection, equipped with the over/under information on sheets along each double point curve. 

We explain a chart on a 2-disk $B$ in a surface diagram $D$ which does not intersect with singularities of $F$. 
We denote the 2-dimensional braid $S\cap p^{-1}(B)$ by $S$. We identify $N(B)$ by $I \times I \times B$. 
Consider the singular set $\mathrm{Sing}(p_1(S))$ of the image of $S$ by the projection $p_1$ to $I \times B$. Perturbing $S$ if necessary, we can assume that $\mathrm{Sing}(p_1(S))$ consists of double point curves, triple points, and branch points. Moreover we can assume that the singular set of the image of $\mathrm{Sing}(p_1(S))$ by the projection to $B$ consists of a finite number of double points such that the preimages belong to double point curves of $\mathrm{Sing}(p_1(S))$. Thus 
the image of $\mathrm{Sing}(p_1(S))$ by the projection to $B$ forms a finite graph $\Gamma$ on $B$ such that the degree of a vertex of $\Gamma$ is either $1$, $4$ or $6$, where we ignore the points in $\partial B$. An edge of $\Gamma$ corresponds 
to a double point curve, and a vertex of degree $1$ (respectively, $6$) 
corresponds to a branch point (respectively, a triple point). 

For such a graph $\Gamma$ obtained from a 2-dimensional braid $S$, we assign orientations and labels to all edges of $\Gamma$ as follows. Let us consider a path $\rho$ in $B$ such that $\rho \cap \Gamma$ is a point $x$ of an edge $E$ of $\Gamma$. Then $S \cap p^{-1} (\rho)$ is a classical $N$-braid with one crossing in $p^{-1}(\rho)$ such that $x$ corresponds to the crossing of the $N$-braid, where $N$ is the degree of $S$. Let $\sigma_{i}^{\epsilon}$ ($i \in \{1,2,\ldots, N-1\}$, 
$\epsilon \in \{+1, -1\}$) be the presentation of $S \cap p^{-1}(\rho)$. Then assign $E$ the label $i$, and the orientation such that 
the normal vector of $\rho$ corresponds (respectively, does not correspond) to the orientation of $E$ if $\epsilon=+1$ (respectively, $-1$), where the normal vector of $\rho$ is a vector $\vec{n}$ such that $(\vec{v}(\rho), \vec{n})$ corresponds to the orientation of $B$ for a tangent vector $\vec{v}(\rho)$ of $\rho$ at $x$. This is the {\it chart of $S$}. 
\\
 
 In general, we define a chart on a surface diagram as follows \cite{N4} (see also \cite{Kamada02}). 
 
\begin{definition} 
Let $N$ be a positive integer. 
A finite graph $\Gamma$ on a surface diagram $D$ is called a {\it chart} of degree $N$ if 
it satisfies 
the following conditions:

\begin{enumerate}[(i)]
\item
The intersection of $\Gamma$ and the singularity set of $D$ consists of a finite number of transverse intersection points of edges of $\Gamma$ and double point curves of $D$, which form vertices of degree $2$.

 \item Every vertex has degree $1$, $2$, $4$, or $6$.
 
 \item Every edge of $\Gamma$ is oriented and labeled by an element of 
       $\{1,2, \ldots, N-1\}$ such that 
       
       \begin{enumerate}
       \item  The adjacent edges around each  of degree $1$, $4$, or $6$  
are oriented and labeled as shown in 
Figure \ref{fig2}, 
where we depict a vertex of degree 1 by a black vertex, and a vertex of degree 
6 by a white vertex, and we call a vertex of degree $4$ a crossing.
\item
The adjacent edges of each vertex of degree 2 are as shown in Figure \ref{fig3}. 
\end{enumerate}
 \end{enumerate}
\end{definition}

\begin{figure}[ht] 
\includegraphics*{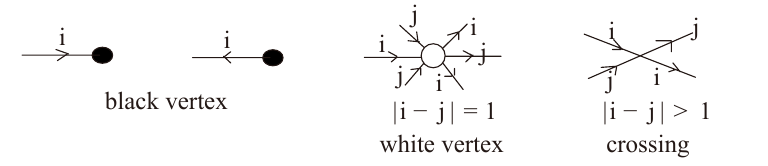}
\caption{Vertices in a chart (a), where $i \in \{1,\ldots,N-1\}$.}
\label{fig2}
\end{figure}
 
\begin{figure}[ht]
\centering\includegraphics*{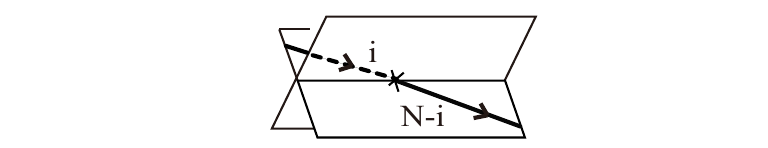}
\caption{A vertex of degree $2$ in a chart (b), where $i \in \{1,\ldots,N-1\}$. For simplicity, we omit the over/under information of each sheet.}
\label{fig3}
\end{figure}

A black vertex (respectively, a white vertex) of a chart corresponds to a branch point (respectively, a triple point) of the 2-dimensional braid presented by the chart.  We call an edge of a chart a {\it chart edge} or simply an {\it edge}. We regard chart edges connected by a vertex of degree 2 as one edge which contains a vertex of degree 2, and we will often omit to mention vertices of degree 2. A chart edge connected with no vertices except crossings (and vertices of degree 2) is called a {\it chart loop} or simply a {\it loop}. A chart is said to be {\it empty} if it is an empty graph. A 2-dimensional braid over a surface-knot $F$ is presented by a chart $\Gamma$ on a surface diagram of $F$ \cite{N4}. We present such a 2-dimensional braid by $(F, \Gamma)$. 
 
\subsection{C-moves}
  
{\it C-moves} are local moves of a chart, consisting of three types: CI-moves, CII-moves, and CIII-moves. 
Let $\Gamma$ and $\Gamma^{\prime}$ be two charts of the same degree on a surface diagram $D$. We say $\Gamma$ and $\Gamma'$ are related by a {\it CI-move}, {\it CII-move} or {\it CIII-move} if there exists a 2-disk $B$ in $D$ such that $B$ does not intersect with the singularities of $D$, and the loop $\partial B$ is in general position with respect to $\Gamma$ and $\Gamma^{\prime}$ and $\Gamma \cap (D-B)=\Gamma^{\prime} \cap (D-B)$, and the following conditions hold true.  
 
(CI) There are no black vertices in $\Gamma \cap B$ nor $\Gamma^{\prime} \cap B$. A CI-move as in Figure \ref{fig5} is called a CI-M1-move, CI-M2-move, CI-M3-move and CI-R2-move respectively; see \cite{CKS} for the complete set of CI-moves. 

(CII) $\Gamma \cap B$ and $\Gamma' \cap B$ are as in Figure \ref{fig5}, where $|i-j|>1$. 

A chart edge connected with a white vertex is called a {\it middle edge} if it is the middle of adjacent three edges around the white vertex with coherent orientations, and it is called a {\it non-middle edge} if it is not a middle edge. Around a white vertex, there are two middle edges and four non-middle edges. 

(CIII) $\Gamma \cap B$ and $\Gamma' \cap B$ are as in Figure \ref{fig5}, where $|i-j|=1$, and the black vertex is connected to a non-middle edge of a white vertex. 
\\

\begin{figure}[ht]
\includegraphics*{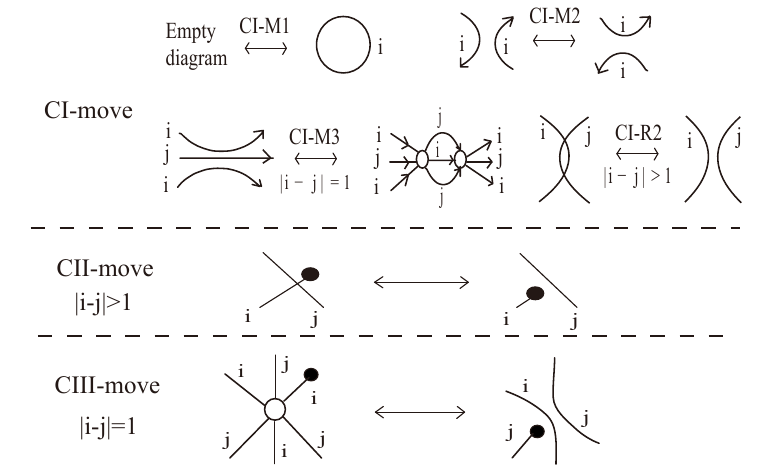}
\caption{C-moves. For simplicity, we omit orientations of some of the edges. Among 7 types of CI-moves, we present 4 types.} 
\label{fig5}
\end{figure}

For charts $\Gamma$ and $\Gamma'$ of the same degree on a surface diagram of a surface-knot $F$, 
their presenting 2-dimensional braids are equivalent if the charts are related by a finite sequence of C-moves \cite{Kamada92, Kamada02}.  

\subsection{Roseman moves}
 
{\it Roseman moves for surface diagrams with charts of the same degree} are defined by the original Roseman moves (see \cite{Roseman}) and local moves as illustrated in Figure \ref{fig4}, where we regard the diagrams for the original Roseman moves as equipped with empty charts. 
 
\begin{figure}[ht]
 \includegraphics*{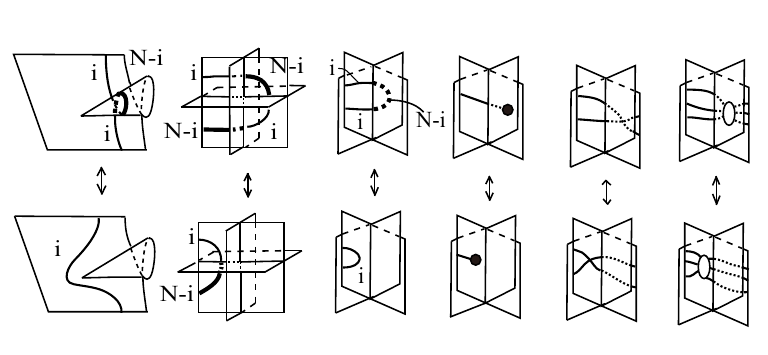}
\caption{Roseman moves for surface diagrams with charts of degree $N$, where $i \in \{1,\ldots,N-1\}$. For simplicity, we omit the over/under information of each sheet, and orientations and labels of chart edges. }
\label{fig4}
\end{figure}
 
 For two surface diagrams with charts, their presenting 2-dimensional braids are equivalent if they are related by a finite sequence of ambient isotopies of $\mathbb{R}^3$ and Roseman moves for surface diagrams with charts of the same degree \cite{N4}.  
 
\section{Two-dimensional braids with repeated pattern}\label{sec3}

We give a precise definition of a 2-dimensional braid with repeated pattern.

\begin{definition}
A 2-dimensional braid over a surface-knot $F$ {\it with repeated pattern with pattern braid $b$} is a 2-dimensional braid such that any closed path in $F$ presents a power of $b$ with respect to the standard framing of $N(F)$. 
\end{definition}

We present by an oriented edge with the label $(1)$ several parallel chart edges presenting $b$. Then a 2-dimensional braid with repeated pattern is presented by several oriented loops with the label $(1)$, and by definition, it is determined from the integers which present the numbers of the loops intersecting the oriented closed curves presenting the generators of $H_1(F; \mathbb{Z})$. 

For simplicity, we present by an oriented edge $E$ with the label $(m)$ $|m|$ copies of parallel oriented edges with the label $(1)$ each of which is equipped with the orientation coherent (respectively, incoherent) with that of $E$, for a non-negative (respectively, non-positive) integer $m$.  A CI-M2-move is presented as in Figure \ref{fig6}(b). We present by a crossing of edges the edges obtained by a  smoothing of the crossing with respect to the orientation (see Figure \ref{fig6}(c)). We call the resulting graph a {\it simplified chart}, or simply a {\it chart} for repeated pattern. From now on, when we treat a chart loop along the core loop or the cocore of a 1-handle, we assume that it is equipped with the orientation coherent with that of the core loop or the cocore. 

\begin{figure}[ht]
\centering
\includegraphics*[height=7cm]{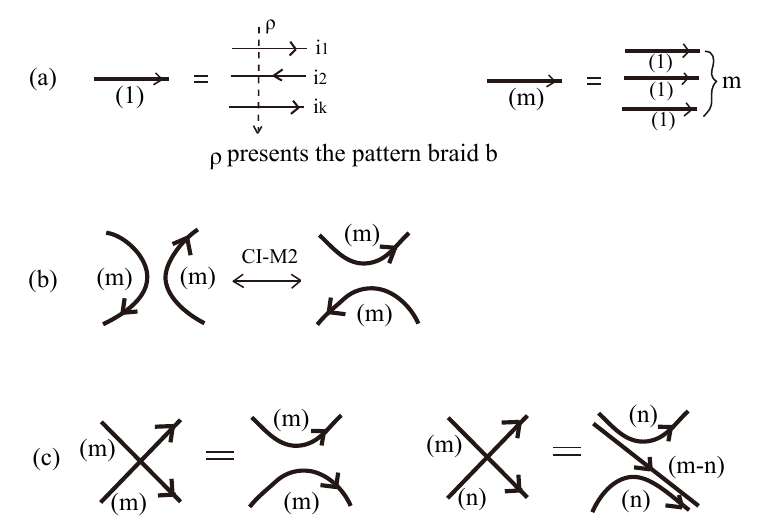}
\caption{Notation of simplified charts for repeated pattern: (a) Notation, (b) a CI-M2-move, (c) Crossings of chart edges.}
\label{fig6}
 \end{figure}

In the first part of this paper, we consider $F$ as a surface-knot which is in the form of the result of 1-handle additions for a surface-knot $F_0$. 
We take the cocores and core loops as the representatives of generators of $H_1(F; \mathbb{Z})$. 
For integers $m$ and $n$, we denote by $h(m,n)$ a 1-handle $h$ with a chart such that the cocore and core loop 
presents the pattern braid to the power $m$ and $n$,  respectively. By CI-moves and Roseman moves, we can assume that $h(m,n)$ is presented by a 1-handle $h$ with a chart loop with the label $(m)$ along the core loop and a chart loop with the label $(n)$ along the cocore. By definition, a 2-dimensional braid over $F$ with repeated pattern is presented by a 2-dimensional braid over $F_0$ with repeated pattern and 1-handles with chart loops. Further, we can assume that all the 1-handles are attached to a fixed small 2-disk in $F$ and there are no chart edges on the 2-disk except those belonging to the 1-handles. 

Now we check that the presentation $h(m,n)$ is well-defined, by using CI-moves, admitting that $h(m,n)$ is presented by a loop with the label $(m)$ along the core loop and a loop with the label $(n)$ along the cocore. By definition of the core loop, it suffices to show Claim \ref{claim1} for the well-definedness.

For a 1-handle $h=B^2 \times [0,1]$, we call $B^2 \times \{0\}$ (respectively, $B^2 \times \{1\}$) the {\it initial end} (respectively, the {\it terminal end}) of $h$. 
For $h$ attached to a 2-disk $B$ and a path $\rho$ in $F$ starting from a point of $B$, we say we {\it move $h$ along $\rho$} when we transform the surface-knot by the ambient isotopy which slides $B$ together with the ends of $h$ along $\rho$.

\begin{claim}\label{claim1}
For a 2-dimensional braid with repeated pattern, a 1-handle with chart loops $h(m,n)$ is invariant under moves along any paths. 
\end{claim}

For simplified charts, when we have a pair of parallel loops with the same label and opposite orientations, by a CI-M2-move we have a loop bounding a disk, and then, by a CI-M1-move we can eliminate the loop (see Figure \ref{fig7}); thus, for parallel loops, we can add up the labels as integers. 

\begin{figure}[ht]
\begin{center}
\includegraphics*[width=12cm]{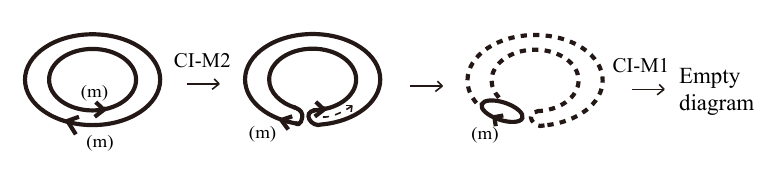}
\end{center}
\caption{Elimination of parallel loops with the same label and opposite orientations.}
\label{fig7}
\end{figure}

\begin{proof}[Proof of Claim \ref{claim1}]
We denote by $H$ a 1-handle with chart loops $h(m,n)$, which is presented by a loop with the label $(m)$ (respectively, $(n)$) along the core loop (respectively, the cocore). By sliding one end if necessary, we can assume both ends of $H$ are on a small 2-disk $B$. It suffices to consider the case when we move $H$ along a path $\rho$ which crosses a chart edge $E$ with the label $(1)$ such that the orientation of $E$ is coherent with the normal of $\rho$. When we move $H$ along $\rho$, by a CI-M2-move, a loop with the label $(1)$ appears along the boundary of the disk where $H$ is attached. Then, by a CI-M2-move, the loop splits into two loops each of which surrounds each end of $H$. Then move the loop surrounding the terminal end of $H$, along $H$, so that it becomes a loop along the cocore, with the label $(-1)$. 
Together with the other loops, we have a loop with the label $(m)$ (respectively, $(n+1-1)=(n)$) along the core loop (respectively, the cocore); see Figure \ref{fig8}. 
\end{proof}

\begin{figure}[ht]
\centering
\includegraphics*[width=13cm]{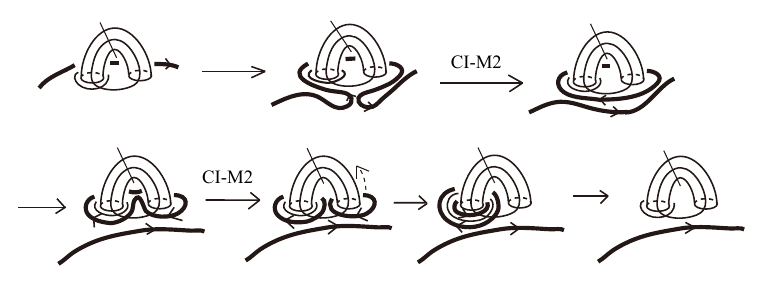}
\caption{Moving a 1-handle with chart loops across a chart edge. For simplicity, we omit the labels and some of the orientations of the chart loops.}
\label{fig8}
\end{figure}

\section{Handle moves}\label{sec4}
 
We use the notation given in Section \ref{sec3}. 
As moves for charts, we consider only CI-M2-moves. As moves for 1-handles, we consider two types of moves as follows (see \cite{Hirose}, see also \cite{Hirose2, Hirose3}). 
 
\noindent
1. {\it Crossing change of tubes.}

We consider a part $B^2 \times I'$ of a 1-handle $B^2 \times I$ for an interval $I' \subset I$ 
with $B^2 \times \partial I'$ fixed, which is called a {\it tube}. We say two tubes form a {\it crossing} if the projected images of their cores in $\mathbb{R}^3$ form a crossing. 
We determine the {\it sign} of a crossing to be positive or negative with respect to the orientations of the cores of the tubes. 

\begin{figure}[ht]
\centering
\includegraphics*{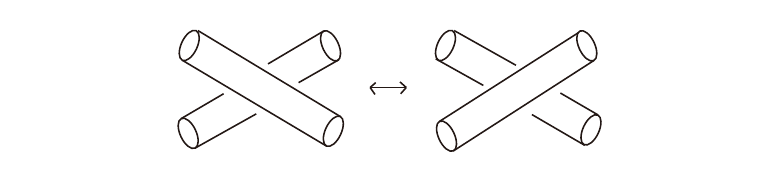}
\caption{Crossing change of tubes.}
\label{fig9}
\end{figure}

\begin{claim}
For two tubes in $\mathbb{R}^4$, a crossing change (see Figure \ref{fig9}) is an equivalent transformation. 
\end{claim}

\begin{proof}
Consider two tubes $h$ and $h'$ which form a crossing. 
We can assume that $h$ is in $\mathbb{R}^3 \times \{0\}$, and $h'$ is in $\mathbb{R}^3 \times \{1\}$. By an ambient isotopy of $\mathbb{R}^4$ rel $\mathbb{R}^4-\mathbb{R}^3 \times[0.5, 1.5]$, we can deform $h' \subset \mathbb{R}^3 \times \{1\}$ to the form $h''\subset \mathbb{R}^3 \times \{1\}$ such that $h$ and $h''$ form a crossing whose sign is opposite to the sign of the original crossing. Hence we have the required result. 
\end{proof}

\noindent
2. {\it Handle slide}. 

We consider a transformation between two 1-handles $h$ and $h'$ such that the terminal end of $h$ slides along the core of $h'$ once. The surface other than near the terminal end of $h$ is fixed. We call this transformation a {\it handle slide} along $h'$ (see Figure \ref{fig10}). 
\\

\begin{figure}[ht]
\centering
\includegraphics*{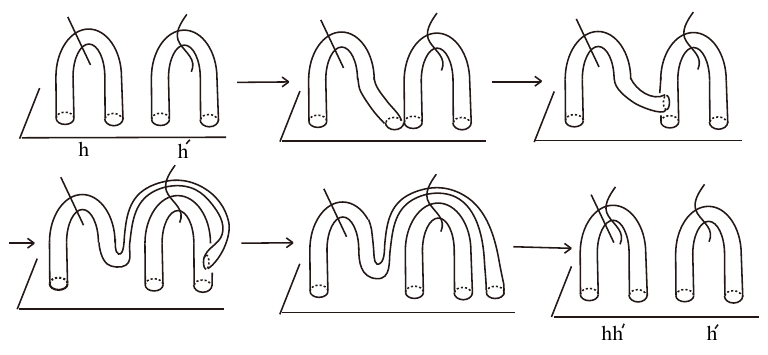}
\caption{Handle slide. In order to indicate that 1-handles may be non-trivial, we draw slashes in the middle of 1-handles.}
\label{fig10}
\end{figure}

Let $h$ and $h'$ be 1-handles presented by $h,h' \in P\backslash G(F_0)/P$, attached to a 2-disk $B$, such that there are no chart edges on $B$ except those belonging to the 1-handles, and let $m,n,m',n'$ be integers. We say 1-handles with chart loops are {\it equivalent} if their presenting 2-dimensional braids are equivalent, and use the notation \lq\lq $\sim$" for equivalence relation. 

\begin{lemma}\label{lem4-2}
We have 
\begin{equation}\label{eq1}
h(m,n) \sim h^{-1}(-m,-n).
\end{equation}
\end{lemma}
 
\begin{proof}
By regarding the orientation reversed path of the core loop of $h$ as the new core loop of $h$, we have the result (see Figure \ref{fig11}). 
\end{proof}

\begin{figure}[ht]
\centering
\includegraphics*[height=3.5cm]{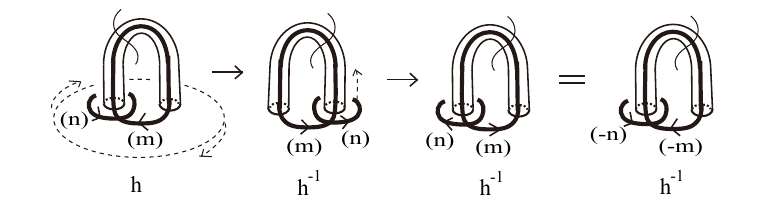}
\caption{$h(m,n) \sim h^{-1}(-m,-n)$.}
\label{fig11}
\end{figure}

\begin{lemma}\label{lem4-3}
We have
\begin{equation}\label{eq2}
h(m,n)\sim h(m,n \pm 2m). 
\end{equation}
\end{lemma}

\begin{proof}
Twist $H=h(m,n)$ once in the middle to make a negative crossing of tubes. Then apply a crossing change of tubes. Then $H$ becomes  $h(m,n+2m)$ (see Figure \ref{fig12}). 
 
The second relation is obtained similarly when we twist $H$ in the opposite way to make a positive crossing. 
\end{proof}

\begin{figure}[ht]
\centering
\includegraphics*[width=12cm]{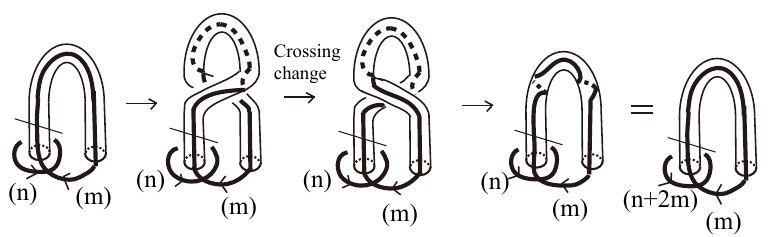}
\caption{$h(m,n)\sim h(m,n + 2m)$. }
\label{fig12}
\end{figure}
 
\begin{remark}\label{rem4-1}
Recall that we equipped a 1-handle with the framing such that the projected image in $\mathbb{R}^3$ has the blackboard framing. Since $1(k,k)$ and $1(k,0)$ ($k \neq 0$) are distinct \cite{Livingston} (see also \cite{Boyle2}), we can see that our framing of a 1-handle is well-defined up to this relation (\ref{eq2}). 
\end{remark}
 
\begin{lemma}\label{lem4-4}
For trivial 1-handles, 
\begin{equation}\label{eq3}
1(m,n) \sim 1(-n,m) \sim 1(n, -m).
\end{equation}
\end{lemma}

\begin{proof}
We assume that the trivial 1-handle with chart loops $H=1(m,n)$ is in $\mathbb{R}^3$. 
Around $H$ attached to a 2-disk $B$, we can assume that $\mathbb{R}^3$ is divided into two regions $R_1$ and $R_2=\mathbb{R}^3-R_1$ such that $H \subset R_1 \cup B$ (see the first figure of Figure \ref{fig13}). We say $H$ is {\it over} (respectively, {\it under}) $B$ if $H$ is in $R_1 \cup B$ (respectively, $R_2 \cup B$). 
Now, push $H$ under $B$. 
This move is equivalent transformation in $\mathbb{R}^4$. 
Then the core loop and cocore of $H$ are exchanged, and $H$ is deformed to $1(-m,n)$ (see Figure \ref{fig13}). The other relation is obtained from $(\ref{eq1})$. 
\end{proof}

\begin{figure}[ht]
\centering
\includegraphics*[width=12cm]{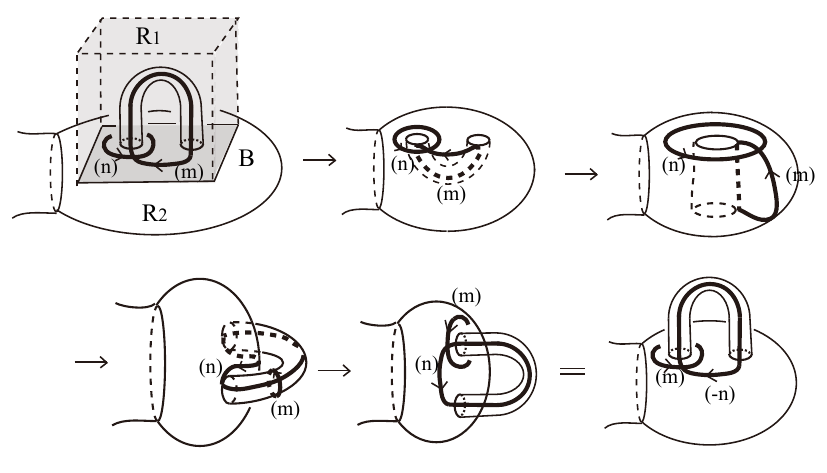}
\caption{$1(m,n) \sim 1(-n,m)$.}
\label{fig13}
\end{figure}
 
\begin{lemma}\label{lem4-5}
We have
\begin{eqnarray}
h(m, n)+ h'(m', n') &\sim& hh'(m, n+n')+ h'(m'-m, n') \label{eq4}\\ 
&\sim & h'^{-1}h(m, n-n') + h'(m'+m, n').\nonumber
\end{eqnarray}

In particular,

\begin{eqnarray}
h(0,n)+ h'(0,n') &\sim& hh'(0, n+n')+ h'(0, n') \label{eq6}\\
&\sim& h'^{-1}h(0, n-n') + h'(0,n'), \nonumber
\end{eqnarray}
\begin{eqnarray}\label{eq7}
h(m,n)+ 1(m',0) &\sim& h(m,n)+ 1(m'\pm m, 0). 
\end{eqnarray}
\end{lemma}

\begin{proof}
We denote by $H_1$ and $H_2$ the first and second 1-handles with chart loops, respectively. 
Present the loop with the label $(m')$ along the core loop of $H_2$ by two loops with the labels $(m)$ and $(m'-m)$. 
Move the loop with the label $(n')$ along the cocore of $H_2$ to the terminal end of $H_2$. 
Then apply a CI-M2-move to the loops with the same label $(m)$ on $H_1$ and $H_2$, and slide the terminal end of $H_1$ along $H_2$, and move the end along a path in $B$ to its original position. Then, applying a CI-M2-move, $H_1$ has a loop with the label $(m)$ along the core loop, and loops with the label $(n)$ and $(n')$ along the cocore and the boundary of the terminal end, respectively, and $H_2$ has a loop with the label $(m'-m)$ along the core loop, and a loop with the label $(n')$ along the cocore. Move along $H_2$ to the cocore, the loop with the label $(n')$ surrounding the terminal end of $H_2$; then, by remarking the orientations and the presentation of 1-handles in $P\backslash G(F_0)/P$, we have the first relation of (\ref{eq4}) (see Figure \ref{fig14}). 
The second relation (\ref{eq4}) is obtained by applying (\ref{eq1}) to $H_1$ before and after applying the first relation of (\ref{eq4}). The other relations follow from (\ref{eq4}) immediately. 
\end{proof}

\begin{figure}[ht]
\centering
\includegraphics*[width=13cm]{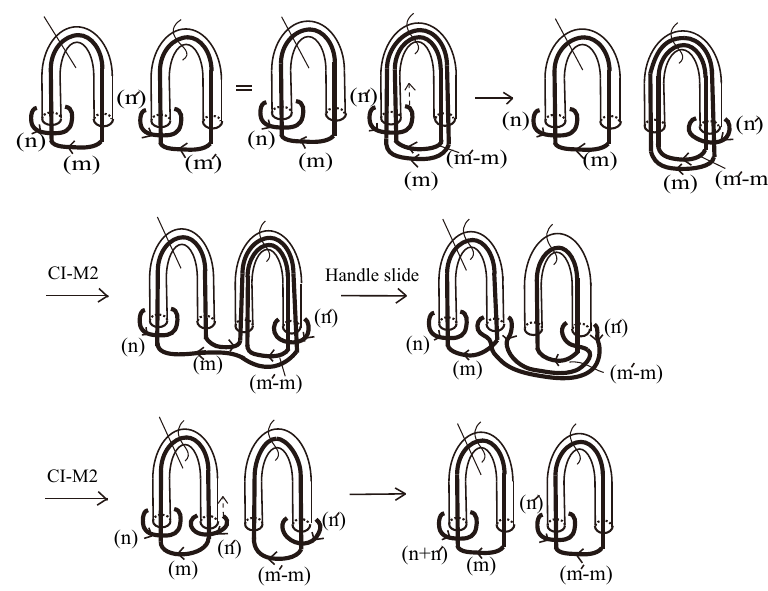}
\caption{$h(m, n)+ h'(m', n') \sim hh'(m, n+n')+ h'(m'-m, n')$.}
\label{fig14}
\end{figure}
 
\begin{lemma}\label{lem4-6}
We have 
\begin{equation}\label{eq9}
h(0,n)+ h'(m',n') \sim h(0,n \pm m') + h'(m',n').
\end{equation}
\end{lemma}

\begin{proof}
We denote by $H_1$ and $H_2$ the first and second 1-handles with chart loops, respectively. Recall that the 1-handles are attached to a 2-disk $B$. 
Move the terminal end of $H_1$ along a path in $B$ across the loop with the label $(m')$ along the core loop of $H_2$. 
By a CI-M2-move, the loop with the label $(m')$ appears along the boundary of the end of $H_1$. Apply a crossing change of tubes for the 1-handles, and move the end of $H_1$ to its original position. Then move the loop with the label $(m')$ along $H_1$ so that it becomes a loop along the cocore. Then $H_1$ is presented by $h(0,n- m')$ and $H_2$ is unchanged; thus we have the second relation (see Figure \ref{fig15}).  
The other relation is obtained by moving the initial end of $H_1$. 
\end{proof}

\begin{figure}[ht]
\centering
\includegraphics*[width=13cm]{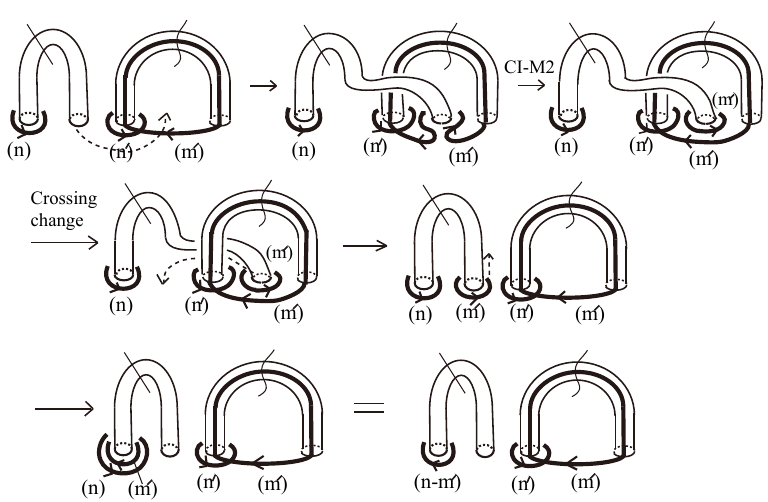}
\caption{$h(0,n)+ h'(m',n') \sim h(0,n - m') + h'(m',n')$.}
\label{fig15}
\end{figure}

\begin{definition}\label{def4-7}
We call these moves (\ref{eq1})--(\ref{eq9}) {\it 1-handle moves}.

\end{definition}

\section{Simplifying 2-dimensional braids with repeated pattern}\label{sec5}

After each deformation in theorems, we denote by $H_j$ the $j$th 1-handle with chart loops. 

\begin{proof}[Proof of Theorem \ref{thm2}]
First we consider the case when some $m_j$ is not zero. 
By changing the indices of 1-handles if necessary, we can assume that $H_1=h_1(m_1, n_1)$ satisfies $m_1 \neq 0$. 
Let $m_{1,1}=\mathrm{gcd} \{m_1,m_2\}$. By applying 1-handle moves (\ref{eq4}) to 1-handles $H_1$ and $H_2$ several times, sliding $H_1$ along $H_2$ or $H_2$ along $H_1$, and applying (\ref{eq1}) if necessary, let us deform $H_1 +H_2$ to the form $h_{1,1}(m_{1,1},n_{1,1}) + h_{2,1}(0, n_{2,1})$ for 1-handles $h_{1,1}, h_{2,1} \in \langle h_1, h_2 \rangle$, and integers $n_{1,1}$ and $n_{2,1}$. Repeat this process to $H_1$ and the other 1-handles $H_j$ ($j=3,\ldots,g$). Then $(F, \Gamma)$ is deformed to 

\begin{equation*}
 h_1' (m, n_1')+ h_{2,2}(0,n_{2,2}) + h_{3,2}(0,n_{3,2}) + \cdots + h_{g,2}(0,n_{g,2}), 
\end{equation*}
where $m=\mathrm{gcd}_{1 \leq j \leq g}\{m_j\}$, and  
$h_1'$, $h_{2,2}, \ldots, h_{g,2} \in \langle h_1, \ldots, h_g \rangle$, and $n_1'$, $n_{2,2} \ldots, n_{g,2}$ are integers. The other case when all $m_j$ are zero is included in this result. 

Then, apply 1-handle moves (\ref{eq6}) to the second 1-handle $H_2$ and the third 1-handle $H_3$ repeatedly and apply (\ref{eq1}) if necessary, until $H_2+H_3$ is deformed to the form $h_{2,3}(0,n_{2,3}) + h_{3,3}(0,0)$, where $n_{2,3}=\mathrm{gcd}\{n_{2,2}, n_{3,2}\}$. Repeat this process to $H_2$ and $H_j$ ($j=4,\ldots, g$).  
Then $(F, \Gamma)$ is deformed to 
\begin{equation*}
 h_1' (m, n_1')+ h_2'(0,n_2') +\sum_{j=3}^g h_j'(0,0), 
\end{equation*}
where $m=\mathrm{gcd}_{1\leq j\leq g}\{m_j\}$, 
$h_1',\ldots, h_g' \in \langle h_1, \ldots, h_g \rangle  < P\backslash G(F_0)/P$, and $n_1', n_2'$ are integers. Thus we have the required result. 
\end{proof}

\begin{proof}[Proof of Theorem \ref{thm1}]
By Theorem \ref{thm2}, $\sum_{j=1}^g 1(m_j, n_j)$ is deformed to 
\begin{equation*}
 1(m_1', n_1')+ 1(0,n_2') +\sum_{j=3}^g 1(0,0), 
\end{equation*}
where $m_1'=\mathrm{gcd}_{1 \leq j\leq g}\{m_j\}$. 
By applying (\ref{eq1}) if necessary, we can assume that $n_2' \geq 0$. Further, by  (\ref{eq9}), we can assume that $0 \leq n_2' < m_1'$. 

We show that $1(m_1', n_1')+ 1(0, n_2')$ is deformed to the form $1(m', n')+1(0, 0)$ for integers $m'$ and $n'$. When $n_2'=0$, we have the required form.  
When $n_2'>0$, by applying (\ref{eq3}) to $H_2$, $1(m_1', n_1')+ 1(0, n_2')$ is deformed to
$1(m_1', n_1')+ 1(n_2',0)$. Then, by the same argument of the proof of Theorem \ref{thm2} and  (\ref{eq1}) and (\ref{eq9}) if necessary, $1(m_1', n_1')+ 1(n_2',0)$ is deformed to $1(m_2', n_{1,2}')+1(0, n_{2,2}')$, where $m_2'=\mathrm{gcd}\{m_1', n_2'\} <m_1'$ and $0 \leq n_{2,2}'<m_2'$. If $n_{2,2}'=0$, then we have the required form. If $n_{2,2}'>0$, then apply the argument again. By repeating this process, we have $1(m_l', n_{1,l}')+1(0, n_{2,l}')$ with $m_l' \geq 1$ and $n_{2,l}'=0$, or $m_l'=1$ and $0 \leq n_{2,l}'<m_l'$. In both cases, $n_{2,l}'=0$ and we have the required form. 

It remains to show that $1(m',n') \sim 1(k,0)$ or $1(k,k)$ for an integer $k$. 
By applying (\ref{eq2}) and (\ref{eq3}) repeatedly to reduce the absolute value of $m'$ and $n'$, we can deform $1(m',n')$ to $1(k,0)$ or $1(k,k)$ for an integer $k$. 
\end{proof}

\begin{proof}[Proof of Theorem \ref{thm3}]
By applying (\ref{eq7}) to $H_j + H_{g+1}$ and sliding $H_j$ along $H_{g+1}$ for $j=1,\ldots,g$, and applying (\ref{eq1}) to $H_{g+1}$ if necessary, 
$(F, \Gamma)+1(0,0)$ is deformed to 
\begin{equation*}
 \sum_{j=1}^g h_j(m_j, n_j) + 1(m,0),
\end{equation*}
where $m=\mathrm{gcd}_{1 \leq j \leq g}\{m_j\}$. 
By applying (\ref{eq4}) to $H_{g+1} + H_j$  and sliding $H_{g+1}$ along $H_j$ $\frac{m_j}{m}$ times for $j=1,\ldots, g$, we have 
\begin{equation}\label{eq5-1}
\sum_{j=1}^g h_j(0,n_j) + h(m,n), 
\end{equation}
where $h \in \langle h_1, \ldots, h_g\rangle$ and $n=\sum_{j=1}^g \frac{m_j n_j}{m}$. 
Since the order of application of these moves does not effect the result (\ref{eq5-1}), we can see that the presentation of $h$ is independent of the order of the generators $h_1, \ldots, h_g$ (\cite{Boyle1}).  
By (\ref{eq9}), we have 
\begin{equation*}
 \sum_{j=1}^g h_j(0, \tilde{n}_j) + h(m,n),
\end{equation*}
where $\tilde{n}_j \in \{0,\ldots, m-1\}$. 
\end{proof}

\begin{proof}[Proof of Theorem \ref{thm4}]
We denote $\sum_{j=1}^{g}1(m_j, n_j)$ by $S$. 
Since $1(k,k)$ and $1(k,0)$ ($k \neq 0$) are distinct \cite{Livingston} (see also \cite{Boyle2}), and it follows from \cite{Sanderson, CKSS} that addition of trivial 1-handles $1(0,0)$ does not change the type (1) or (2) in Theorem \ref{thm1}, we see that $S+1(0,0)$ has the same type with $S$. Hence we will deform $S+1(0,0)$ by 1-handle moves.

By (\ref{eq1}) and (\ref{eq3}), $H_j$ can be presented by $1(m_j, n_j)$, $1(-m_j, -n_j)$, $1(n_j, -m_j)$, or $1(-n_j, m_j)$ $(j=1,\ldots,g)$. Thus, 
by applying (\ref{eq7}) to $H_j + H_{g+1}$ and sliding $H_j$ along $H_{g+1}$ $(j=1,\ldots,g)$, 
$S + 1(0,0)$ is deformed to 
\begin{equation*}
 \sum_{j=1}^g 1(m_j, n_j) + 1(\mathrm{gcd},0),
\end{equation*}
where $\mathrm{gcd}=\mathrm{gcd}_{1 \leq j \leq g}\{m_j, n_j\}$.  
 
By applying (\ref{eq4}) to $H_{g+1}+ H_j$ $\frac{m_j}{\mathrm{gcd}}$ times for $j=1,\ldots, g$, we have
\begin{equation*}
 \sum_{j=1}^g 1(0, n_j) +1(\mathrm{gcd}, \sum_{j=1}^g \frac{m_jn_j}{\mathrm{gcd}}). 
\end{equation*}
Since $n_j$ is divided by $\mathrm{gcd}$, by applying (\ref{eq9}) to $H_j + H_{g+1}$, we have 
\begin{equation*}
 \sum_{j=1}^g 1(0, 0)+ 1(\mathrm{gcd},\sum_{j=1}^g\frac{m_jn_j}{\mathrm{gcd}}). 
\end{equation*}
Since $\frac{\sum_{j=1}^g m_j n_j}{\mathrm{gcd}}$ is divided by $\mathrm{gcd}$, by (\ref{eq2}) we see that $S + 1(0,0)$ is equivalent to 
  $1(\mathrm{gcd}, \mathrm{gcd})+\sum_{j=2}^g1(0,0)$ if $\frac{\sum_{j=1}^g m_jn_j}{(\mathrm{gcd})^2}$ is odd, and 
$1(\mathrm{gcd}, 0)+\sum_{j=2}^g1(0,0)$ if $\frac{\sum_{j=1}^g m_jn_j}{(\mathrm{gcd})^2}$ is even, which implies the required result.
\end{proof}

\section{Unbraiding 2-dimensional braids without branch points}\label{sec6}

\begin{proof}[Proof of Theorem \ref{thm5}]
The chart $\Gamma$ consists of several loops with the label $(1)$. 
Applying a CI-M2-move to a chart loop $E$ and $H=1(1,0)$, and then sliding an end of $H$ along $E$, the union of $E$ and $H$ is deformed to $H=h'(1,n)$ for a 1-handle $h'=\alpha \rho \alpha^{-1}$ and an integer $n$, where $\rho$ is a closed path in $\partial N(F)$ with a base point $x'$ such that the projection of $\rho$ to $F$ is $E$, and $\alpha$ is a path in $\partial N(F)$ connecting the base point $x$ of the knot group $G(F)$ and $x'$; thus we can eliminate $E$. Repeat this process to every loop and $H$, and the chart loops gather on $H$ with presentation $h(1, n')$ for a 1-handle $h$ attached to $F$ and an integer $n'$.  Then, by applying 1-handle moves  (\ref{eq2}), and applying (\ref{eq1}) if necessary, $(F, \Gamma) +1(1,0)$ is deformed to $(F, \emptyset) + h(1, \delta)$, where $\delta \in \{0, 1\}$. 
\end{proof}

Now, we consider $(F, \Gamma)$ for any surface-knot $F$ and a chart $\Gamma$ without black vertices. Let $N$ be the degree of $\Gamma$. 

Among the 6 edges connected with a white vertex, we will call {\it diagonal edges} a pair of edges between which there are two edges on each side; there are three pairs of diagonal edges, one consisting of middle edges, and the other two consisting of non-middle edges. 
We denote by $h(a,b)$ a 1-handle $h$ with a chart without black vertices, attached to 
$F$, such that the chart is contained in the union of regular neighborhoods of the core loop and cocore in $F$, and the cocore and the core loop present braids $a$ and $b$, respectively. Note that $a$ and $b$ are commutative, and for any commutative braids $a$ and $b$, $h(a,b)$ is well-defined \cite{N}. 
In this paper, we consider 1-handles with chart loops, in the form $h(\sigma_i, e)$, $h(\sigma_i, \sigma_j^{\epsilon})$, and $h(e,e)$, where $i,j \in \{1.\ldots,N-1\}, |i-j|>1$ and $\epsilon \in \{+1, -1\}$. 

\begin{remark}\label{rem6-1}
By the following lemma, we see that a 1-handle $H=h(a,b)$ is not determined from the presentation; a 2-dimensional braid with repeated pattern is a special case where the presentation is well-defined. 
In order to make the presentation well-defined, it is necessary to determine the region which contains the chart associated with $H$, and further it is necessary to assign the place where we attach $H$. 
\end{remark}

\begin{lemma}\label{lem6-2}
A 1-handle $h(e,a)$ becomes $h(e,bab^{-1})$ after crossing chart edges presenting $b$.
In particular, a 1-handle $h(e,e)$ can be moved anywhere.
\end{lemma}

\begin{proof}
It suffices to show the case when $b=\sigma_i$. By the same argument as in the proof of Claim \ref{claim1}, we have the result. 
\end{proof}

Before the proofs of Theorems \ref{thm6} and \ref{thm7}, we prepare several lemmas. We denote by $\sum_{j=1}^{N-1} 1(\sigma_j,e)$ a set of 1-handles with chart loops, attached to a small 2-disk $B_0$ in $F$ such that $B_0$ is disjoint with the chart on $F$. 

\begin{lemma}\label{lem6-3}
A set of 1-handles $\sum_{j=1}^{N-1} 1(\sigma_j,e)$ can be moved anywhere. 
\end{lemma}

\begin{proof}
It suffices to show that this set of 1-handles can move across a chart edge with the label $i$. Let us denote by $D_1$ and $D_2$ the regions divided by the edge such that the 1-handles are attached to $D_1$. 
Apply a CI-M2-move on the edge and $H=1(\sigma_i, e)$. Then there is a path $\rho$ from $D_1$ to $D_2$ which crosses no edges. 
Then we can import the other 1-handles to $D_2$ by moving them along $\rho$. Applying a CI-M2-move we can move $H$ to the other region $D_2$ (see Figure \ref{fig16-2}). 
\end{proof}

\begin{figure}[ht]
\centering
\includegraphics*[width=13cm]{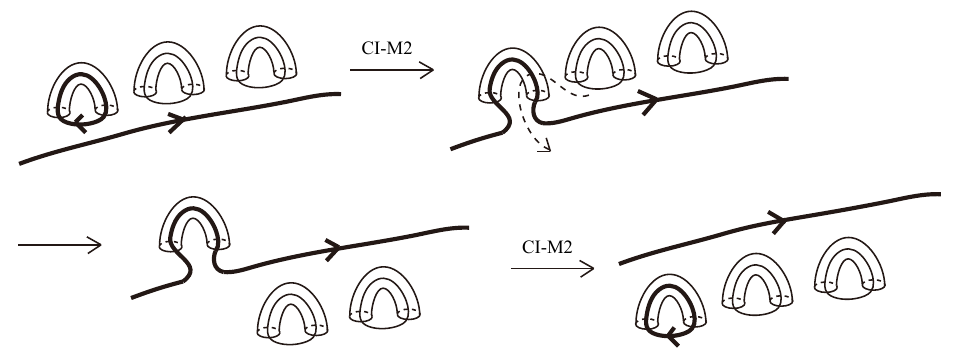}
\caption{Moving a set of 1-handles with chart loops across a chart edge. For simplicity, we omit the labels of the chart edges and the orientations of chart loops.}
\label{fig16-2}
\end{figure}

Remark that similar result holds true for non-trivial 1-handles $\sum_{j=1}^{N-1} h_j(\sigma_j,e)$, and $h_j(\sigma_j,e)$ can be replaced by $h_j(\sigma_j, b)$ for any $N$-braid $b$ which commutes with $\sigma_j$.  

\begin{lemma}\label{lem6-4}
Together with $\sum_{j=1}^{N-1} 1(\sigma_j, e)$, a 1-handle $1(e,e)$ can be transformed to $1(\sigma_i, e) $ for any $i \in \{1,\ldots, N-1\}$:

\[
\sum_{j=1}^{N-1} 1(\sigma_j,e) + 1(e,e) \sim \sum_{j=1}^{N-1} 1(\sigma_j,e) + 1(\sigma_i, e).
\]
\end{lemma}

\begin{proof}
By applying a 1-handle move similar to (\ref{eq7}) in Lemma \ref{lem4-5} to $1(\sigma_i, e) + 1(e,e)$, we have the result. 
\end{proof}

\begin{lemma}\label{lem6-5}
When we have a 1-handle $H=1(\sigma_i,e)$ near a non-middle edge (respectively, an edge) with the label $i$ of a white vertex (respectively, a crossing), by sliding one end of $H$ along the diagonal edges, we can have the vertex on $H$ (see Figures \ref{fig16} and \ref{fig17}). 
\end{lemma}

For the deformation as in Lemma \ref{lem6-5}, we say that we {\it make a bridge} over a white vertex or a crossing by a 1-handle with a chart loop. 

\begin{figure}[ht]
\centering
\includegraphics*[width=13cm]{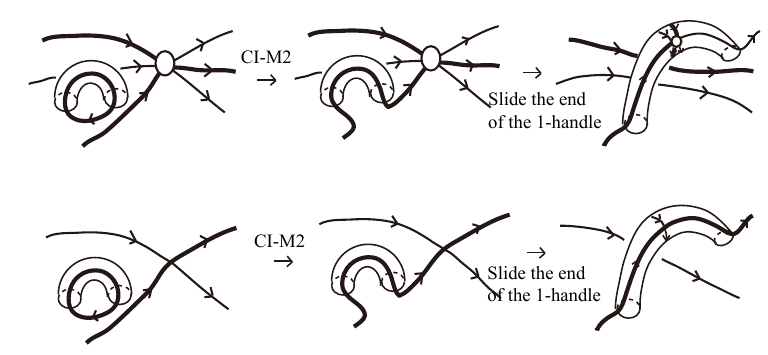}
\caption{Making a bridge over a white vertex or a crossing. For simplicity, we omit the labels of the chart edges.}
\label{fig16}
\end{figure}
 
\begin{figure}[ht]
\centering
\includegraphics*[width=13cm]{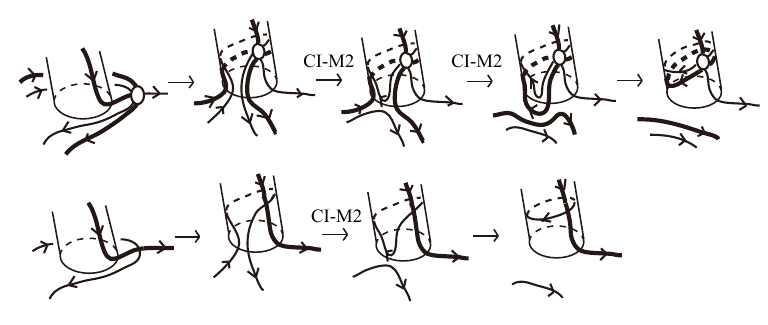}
\caption{Sliding an end of a 1-handle along diagonal chart edges. For simplicity, we omit the labels of the chart edges.}
\label{fig17}
\end{figure}
 
\begin{proof}[Proof of Lemma \ref{lem6-5}]
By applying a CI-M2-move to the chart loop of $H$ and the non-middle edge of the white vertex or the  edge of the crossing, and sliding an end of $H$ along the diagonal edges and applying CI-M2-moves, we have the result (see Figures \ref{fig16} and \ref{fig17}). 
\end{proof}
  
\begin{lemma}\label{lem6-6}
When we have a white vertex as in the left figure of Figure \ref{fig18} on a 1-handle, an addition of another 1-handle near the vertex induces the orientation reversal of the edges around the vertex. 
\end{lemma}
 
\begin{proof}
When we have such a white vertex, let $i$ be the label of a non-middle edge $E$ which is along the cocore and is connected with the vertex again as a middle edge. Add $H=1(\sigma_i,e)$ near $E$ and have the white vertex on $H$ by making a bridge. Then the edges around the white vertex on $H$ become the orientation-reversed ones to the original edges; see Figure \ref{fig18}.
\end{proof}

\begin{figure}[ht]
\centering
\includegraphics*[width=12cm]{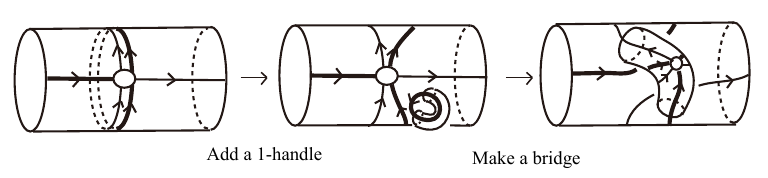}
\caption{Reversal of orientations of the edges around a white vertex on a 1-handle. For simplicity, we omit the labels of the chart edges.}
\label{fig18}
\end{figure}

We consider 1-handles in the form $h(\sigma_i, b)$, where $b$ is a braid commutative with $\sigma_i$. In particular, we consider 1-handles consisting of chart loops containing crossings. 

\begin{lemma}\label{lem6-7}
For 1-handles $h_1,\ldots, h_{N-1}$, and braids $b_1, \ldots, b_{N-1}$ such that $b_j$ commutes with $\sigma_j$ $(j=1,\ldots,N-1)$,  
we assume that $b_i$ has the presentation $\prod_{k} \sigma_{i_k}^{\epsilon_k}$ such that $|i-i_k|>1$ and $\epsilon_k \in \{+1, -1\}$. Then, 
we have 
\[
\sum_k 1(e,e)+\sum_{j=1}^{N-1}h_j(\sigma_j, b_j)\sim \sum_k 1(\sigma_i, \sigma_{i_k}^{\epsilon_k})+h_i(\sigma_i, e)+ \sum_{j\neq i}^{N-1}h_j(\sigma_j, b_j).  
\]
\end{lemma}

\begin{proof}
We deform the right form to the left form. Let $\prod_{k=1}^m \sigma_{i_k}^{\epsilon_k}$ be the presentation of $b_i$. 
 By applying a CI-M2-move to $H=h_i(\sigma_i, e)$ and the last 1-handle $H'=1(\sigma_i, \sigma_{i_m}^{\epsilon_m})$ in $\sum_{k=1}^m 1(\sigma_i, \sigma_{i_k}^{\epsilon_k})$, and sliding $H$ along $H'$ as in Lemma \ref{lem4-5}, and by CI-M2-moves as in Figure \ref{fig17}, $H+H'$ is deformed to $h_i(\sigma_i, \sigma_{i_m}^{\epsilon_m})+1(e, \sigma_{i_m}^{\epsilon_m})$. 
Repeat this process to $H$ and 1-handles $1(\sigma_i, \sigma_{i_{m-k+1}}^{\epsilon_{m-k+1}})$ in $\sum_{k=1}^m 1(\sigma_i, \sigma_{i_k}^{\epsilon_k})$ for $k=1,2,\ldots, m$. Then, 
$\sum_{k=1}^m 1(\sigma_i, \sigma_{i_k}^{\epsilon_k})+h_i(\sigma_i, e)$ 
is deformed to 
\[
\sum_{k=1}^m 1(e, \sigma_{i_k}^{\epsilon_k})+h_i(\sigma_i, b_i). 
\]
Then, by moving an end of each $1(e, \sigma_{i_k}^{\epsilon_k})$ across the chart loop along the core loop of $h_{i_k}(\sigma_{i_k}, b_{i_k})$ in $\sum_{j\neq i}^{N-1}h_j(\sigma_j, b_j)$ as in Lemma \ref{lem4-6}, we can eliminate the loop on each $1(e, \sigma_{i_k}^{\epsilon_k})$, and the required result follows. 
\end{proof}
 
\begin{proof}[Proof of Theorem \ref{thm6}]
We prove the second relation. 
By Lemmas \ref{lem6-2}, \ref{lem6-3} and \ref{lem6-4}, by an addition of 1-handles $\sum_{j=1}^{N-1} 1(\sigma_j, e)$ and a 1-handle $H=1(e,e)$ to a 2-disk $B_0$, we can move $H$ to any 2-disk $B$ in $F$ such that $H$ attached to $B$ has the presentation $1(\sigma_i, e)$ for any $i \in \{1,\ldots,N-1\}$. 
For this deformation, we say we {\it send $1(\sigma_i, e)$ to $B$}.  
 
First we add to $(F, \Gamma)$ 1-handles with chart loops attached to $B_0$ as follows:
\begin{equation*}
(F, \Gamma)+\sum 1(e,e)+\sum_{j=1}^{N-1} 1(\sigma_j, e), 
\end{equation*} 
for a large number of 1-handles $1(e,e)$. We denote the result of 1-handle additions by $S$. 

We eliminate white vertices in $\Gamma$ as follows. 
Send a 1-handle $H=1(\sigma_i,e)$ to near a non-middle edge with the label $i$ of a white vertex. Make a bridge to have the white vertex on $H$. Then slide one end of $H$ along the diagonal edges. If the end comes to a non-middle edge of another white vertex, then make a bridge and have the vertex on $H$ again. If the end comes to a middle edge of another white vertex or a crossing, then send another 1-handle to near another non-middle edge of the white vertex or an edge of the crossing, and make a bridge to let $H$ pass. Repeat this process, until the end of $H$ comes back near the other end. Then, on $H$, there are only white vertices as vertices. Apply a CI-M3-move (see Figure \ref{fig5}) to adjacent white vertices. If the move cannot be applied, then add another 1-handle to reverse the orientations of the edges around one of the white vertices so that we can apply the move; thus we can eliminate white vertices on 1-handles, and the resulting 1-handles are in the form $h(\sigma_i, e)$ for a 1-handle $h$ attached to $F$. Repeat this process, until we eliminate all the white vertices; thus, $S$ is deformed to 

\begin{equation}\label{eq6-6}
(F, \Gamma')+\sum_\mu H_\mu +\sum 1(e,e)+\sum_{j=1}^{N-1} 1(\sigma_j, e), 
\end{equation} 
where $\Gamma'$ is a chart which has no white vertices, and 
$H_\mu=h_\mu(\sigma_i, e)$ for a 1-handle $h_\mu$ attached to a 2-disk $D_\mu$, and $i \in \{1,\ldots, N-1\}$, and the other 1-handles are attached to $B_0$. 

Next we eliminate all chart loops on $F+\sum_\mu h_\mu$. 
Since we have $\sum_{j=1}^{N-1} 1(\sigma_j, e)$ attached to $B_0$,  apply a CI-M2-move to a chart loop $E$ nearest $\sum_{j=1}^{N-1} 1(\sigma_j, e)$ and $H=1(\sigma_i,e)$ in $\sum_{j=1}^{N-1} 1(\sigma_j, e)$, where $i$ is the label of $E$. Then, by sliding one end of $H$ along $E$ as in the proof of Theorem \ref{thm5}, and by CI-M2-moves as in Figure \ref{fig17}, the union of $E$ and $H$ is deformed to a 1-handle with presentation $h(\sigma_i, b)$, where $h$ is a 1-handle attached to $B_0$, and $b$ is a braid which commutes with $\sigma_i$, representing the crossings on $E$.  Thus we eliminate $E$. Repeat this process to every chart loop in $\Gamma'$ and $H_\mu$, until $S$ is deformed to 
    
\[
(F, \emptyset)+\sum_\mu H'_\mu+\sum1(e,e)+\sum_{j=1}^{N-1}h_j(\sigma_j, b_j), \]
where $H'_\mu=h_\mu(e,e)$ for a 1-handle $h_\mu$ attached to a 2-disk $D_\mu$, and $h_1, \ldots,h_{N-1}$ are 1-handles attached to $B_0$, $b_j$ is a braid which commutes with $\sigma_j$ ($j=1,\ldots,N-1$), and the other trivial 1-handles are attached to $B_0$. 

Now, ignoring charts, we have deformed $F+\sum_\mu 1$ to the form $F+\sum_\mu h_\mu$. Hence, $(F+\sum_\mu h_\mu, \emptyset)$ can be deformed to $(F+\sum_\mu 1, \emptyset)$. 
Thus, by deforming $F+\sum_\mu h_\mu$ by a reverse deformation to recover the original trivial 1-handles, we have 
\begin{equation}\label{eq6-1}
(F, \emptyset)+\sum1(e,e)+\sum_{j=1}^{N-1}h_j(\sigma_j, b_j).  
\end{equation}

Apply the deformation as in Lemma \ref{lem6-7} to $\sum1(e,e)+\sum_{j=1}^{N-1}h_j(\sigma_j, b_j)$ $N-1$ times, and we have the required result. 
\end{proof}

\begin{definition}\label{def6-7}
For a chart $\Gamma$ of degree $N$, we say that a crossing consisting of diagonal edges with the label $i$ and $j$ ($i<j,\ j-i>1$) is of {\it type} $(i,j)$, and a crossing of type $(i,j)$ has the {\it sign} $+1$ (respectively, $-1$) if the normal of the edge with the label $i$ is coherent (respectively, incoherent) with the orientation of the edge with the label $j$. The {\it algebraic sum of the number of crossings in $\Gamma$ of type $(i,j)$}, denoted by $c_{\mathrm{alg},i,j}(\Gamma)$, is the sum of the signs of crossings of type $(i,j)$ in $\Gamma$, and we define $c_{\mathrm{alg}}(\Gamma)$ by $c_{\mathrm{alg}}(\Gamma)=\sum_{1\leq i<j\leq N-1}|c_{\mathrm{alg},i,j}(\Gamma)|$. 
\end{definition}

\begin{proof}[Proof of Theorem \ref{thm7}]
We show the second relation. 
By Theorem \ref{thm6}, it suffices to show that $1(\sigma_i, \sigma_j) + 1(\sigma_i, \sigma_j^{-1})$ ($|i-j|>1$) is equivalent to trivial 1-handles with chart loops without crossings. We denote by $H_1$ and $H_2$ the first and the second 1-handles, respectively. By sliding $H_1$ on $H_2$ as in Lemma \ref{lem4-5}, $H_1 + H_2$ is deformed to $1(\sigma_i, \sigma_j\sigma_j^{-1}) +1(e,\sigma_j^{-1})$: $H_1$ has two crossings and $H_2$ has a loop without crossings. Since $\sigma_j^{-1}\sigma_j=e$ in $B_N$, $H_1$ is equivalent to $1(\sigma_i, e)$, but we will show this by using C-moves. Since the crossings in $H_1$ have opposite signs, by a CI-M2-move, we have a loop bounding a disk with the label $j$ on $H_1$. By a CI-R2-move (see Figure \ref{fig5}), we can eliminate the crossings on the loop, and by a CI-M1-move, we can eliminate the loop itself, and the resulting 1-handle is $h(\sigma_i, e)$ (see also Figure \ref{fig7}). Thus $H_1+H_2$ is equivalent to $1(\sigma_i,e)+1(e, \sigma_j^{-1})$, trivial 1-handles with chart loops without crossings, and the result follows. 
\end{proof}

\begin{proof}[Proof of Proposition \ref{prop1}]
The inequality $u_w(F, \Gamma) \leq u(F, \Gamma)$ is obvious. By the proof of Theorem \ref{thm7}, the other inequality $u(F, \Gamma) \leq u_w(F, \Gamma)+c_{\mathrm{alg}}(\Gamma)$ holds true. 
\end{proof}

A chart edge is called a {\it free edge} if it is connected with two black vertices at its end points. An addition of 1-handles with chart loops is similar to an addition of free edges (see \cite{Kamada99}, see also \cite[Chapter 31]{Kamada02}).
 
\begin{proof}[Proof of Proposition \ref{prop2}]
We show that the inequality $u_w(F, \Gamma) \leq w(\Gamma)+2c(\Gamma)+N-1$ follows from the proof of Theorem \ref{thm6}. For each crossing, add two 1-handles  $H$ and $H'$ to make double bridges such that $H$ is attached to $F$ and $H'$ is attached to $H$, so that we have the crossing on $H'$ and the edges which formed the crossing were separated on $F$ and $H$ as simple edges without crossings (see Figure \ref{fig18-2}). Then, when we move ends of 1-handles to gather white vertices, we can move them along edges without crossings. Thus, we use $2c(\Gamma)$ 1-handles. From now on, we fix these 1-handles in these forms. 

Add a set of 1-handles $\sum_{j=1}^{N-1}1(\sigma_j,e)$ and $w(\Gamma)$ 1-handles in the form $1(e,e)$, attached to a 2-disk. 
Then, for each white vertex, send a 1-handle and make a bridge. Since $\Gamma$ contained no black vertices, now, for every white vertex $W$, there is an embedded circle containing $W$, which consists of non-middle edges connecting white vertices. 
Let $E$ be such an embedded circle. Since $E$ consists of non-middle edges, slide an end of one of the added 1-handles along $E$ to gather all white vertices on $E$ on the 1-handle. Note that since the diagonal edges forming $E$ are labeled by odd integers and even integers in turn, the number of the white vertices is even. Repeat this process to every such an embedded circle, so that we have $m$ 1-handles with gathered white vertices, and $w(\Gamma)-m$ 1-handles with chart loops without crossings, and $\sum_{j=1}^{N-1}1(\sigma_j,e)$. 
We need at most $w(\Gamma)/2$ 1-handles to change the orientations of the edges around white vertices to remove them by CI-M3 moves. These 1-handles can be obtained by recycling the other $w(\Gamma)-m$ 1-handles with chart loops without crossings, by using $\sum_{j=1}^{N-1}1(\sigma_j,e)$. Note that on each 1-handle gathering the white vertices, there are at least two white vertices, hence $m \leq w(\Gamma)/2$; this implies that $w(\Gamma)/2\leq w(\Gamma)-m$, and we see that we have enough 1-handles. Thus we remove all the white vertices. Then use $\sum_{j=1}^{N-1}1(\sigma_j,e)$ to eliminate the chart loops and have the form (\ref{eq6-1}) in the proof of Theorem \ref{thm6}. Thus, in total we use at most $2c(\Gamma)+w(\Gamma)+N-1$ 1-handles to make $\Gamma$ an empty chart. Now we have  $2c(\Gamma)+w(\Gamma)$ 1-handles in the form $1(e,e)$. Since we need $c(\Gamma)$ such 1-handles to apply deformations as in Lemma \ref{lem6-7} to obtain the required form, we see that we added enough 1-handles and $u_w(F, \Gamma) \leq w(\Gamma)+2c(\Gamma)+N-1$. 
\end{proof}
\begin{figure}[ht]
\centering
\includegraphics*[height=5cm]{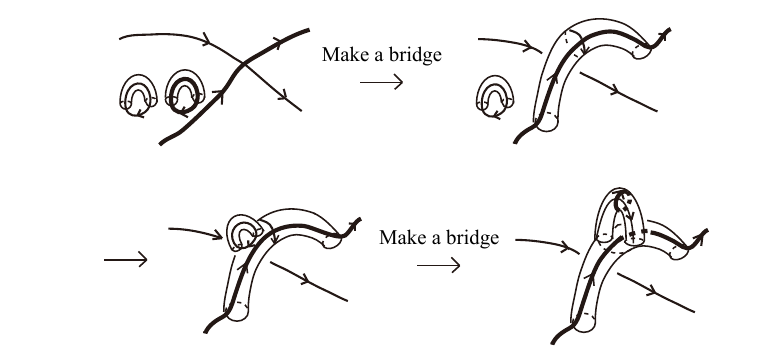}
\caption{Making double bridges over a crossing. For simplicity, we omit the labels of the chart edges.}
\label{fig18-2}
\end{figure}
 
\section{Unbraiding 2-dimensional braids with branch points}\label{sec7}

We consider $(F, \Gamma)$ for any surface-knot $F$ and any chart $\Gamma$. 

\begin{remark}\label{rem6-8}
Let $(F, \Gamma)$ be a 2-dimensional braid for any surface-knot $F$ and any chart $\Gamma$. 
Then Theorems \ref{thm6} and \ref{thm7} hold true when we change the resulting $(F, \emptyset)$ to $(F, \Gamma_0)$, where $\Gamma_0$ is an unknotted chart, and Propositions \ref{prop1} and \ref{prop2} hold true with unchanged. 
\end{remark}

\begin{proof}
It suffices to show that we can discuss the same argument as in the proof of Theorem \ref{thm6}. In order to show this, it suffices to see the step when we move an end of a 1-handle $H$ to gather white vertices on $H$, in particular when the diagonal edges form an arc whose endpoints are black vertices, along which we move an end of $H$.

In this case, move the both ends of $H$ along diagonal edges of white vertices, by making bridges to avoid passing crossings and middle-edges, until we gather white vertices and two black vertices on $H$. 
Since the diagonal edges connected with the black vertices are all non-middle edges, by CIII-moves (see Figure \ref{fig5}), we can eliminate all the white vertices on $H$, and $H$ has a free edge and chart loops along the cocore. 
Thus each $H_\mu$ in (\ref{eq6-6}) in the proof of Theorem \ref{thm6} becomes a 1-handle with a free edge and chart loops along the cocore. 

Hence, we can discuss the same argument as in the proof of Theorem \ref{thm6}, and the resulting chart $\Gamma_0$ on $F$ does not contain white vertices, crossings or chart loops; which implies that $\Gamma_0$ is a chart consisting of free edges, an unknotted chart. 
Hence the similar results as in Theorems \ref{thm6} and \ref{thm7} hold true when we change the resulting $(F, \emptyset)$ to $(F, \Gamma_0)$, where $\Gamma_0$ is an unknotted chart. 
By the same argument as in the proofs of Propositions \ref{prop1} and \ref{prop2}, the same inequalities in the propositions hold true. 
\end{proof}

Before the proof of Theorem \ref{thm12}, we prepare the following lemma. 

\begin{lemma}\label{lem7-2}
We denote by $f_i$ a free edge with the label $i$. 
For 1-handles $h_1, \ldots, h_{N-1}$, we have  

\[
(F, f_i)+\sum_{k=1}^{N-1}h_k(\sigma_k, e)\sim (F, f_j)+\sum_{k=1}^{N-1}h_k(\sigma_k, e), 
\]
for any $i,j \in \{1,\ldots, N-1\}$. 

\end{lemma}

\begin{proof}
It suffices to show for the case when $|i-j|=1$. 
When we have a free edge $f_i$, move it across the chart loop with the label $j$ along the core loop of $h_j(\sigma_j, e)$ to add a loop with the label $j$ surrounding $f_i$ (see Figure \ref{fig19}). Then, by CIII-moves, $f_i$ surrounded by the loop is deformed to a free edge $f_j$  surrounded by a loop with the label $i$ (see Figure \ref{fig20}). Then, move the resulting chart across the chart loop with the label $i$ along the core loop of $h_i(\sigma_i, e)$ to remove the loop. Thus $f_i$ is deformed to $f_j$. 
\end{proof}

\begin{figure}[ht]
\centering
\includegraphics*[width=13cm]{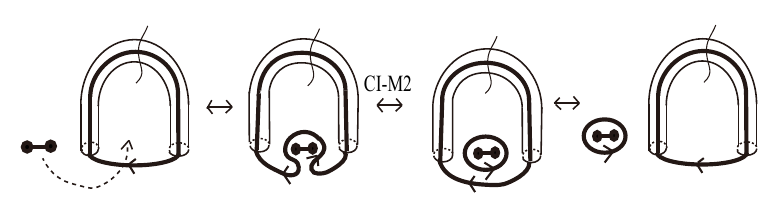}
\caption{Moving a free edge across a chart loop along the core loop of a 1-handle. For simplicity, we omit the labels of chart edges and the orientation of the free edge.}
\label{fig19}
 \end{figure}

\begin{figure}[ht]
\centering
\includegraphics*[width=13cm]{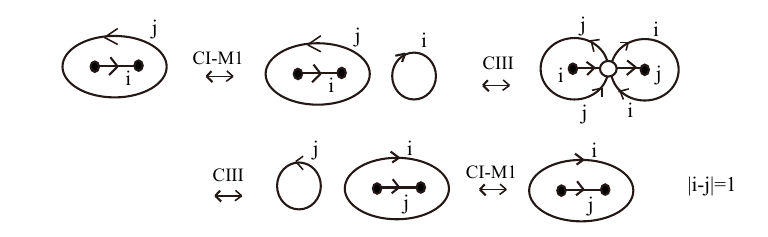}
\caption{Changing the labels of free edges surrounded by a loop.}
\label{fig20}
 \end{figure}

\begin{proof}[Proof of Theorem \ref{thm12}]
By the result similar to Theorem \ref{thm6}, by an addition of 1-handles $\sum_{j=1}^{N-1}1(\sigma_j, e)$ and finitely many $1(e,e)$, to a fixed 2-disk in $F$, $(F, \Gamma)$ is deformed to 
\begin{equation*}
(F, \Gamma_0) +\sum_{j=1}^{N-1}h_j(\sigma_j, e)+\sum_\lambda 1_\lambda,
\end{equation*}
where $\Gamma_0$ is an unknotted chart, $h_1, \ldots, h_{N-1}$ are 1-handles attached to $F$, and $1_\lambda=1(\sigma_i, \sigma_j^{\epsilon})$ or $1(e,e)$ ($i,j \in \{1.\ldots,N-1\}, |i-j|>1$ and $\epsilon \in \{+1, -1\}$). 
The unknotted chart $\Gamma_0$ consists of $b(\Gamma)/2$ free edges. Since $b(\Gamma) \geq 2(N-1)$, $\Gamma_0$ consists of at least $N-1$ free edges. 
By Lemma \ref{lem7-2}, we can deform $\Gamma_0$ so that $\Gamma_0$ contain free edges of all labels in $\{1, \ldots, N-1\}$. 

Then, by applying a CI-M2-move to a chart loop $E$ and a free edge $f$ of the same label and applying a CII-move if necessary, let us deform the union of $E$ and $f$ to $f$; thus we  eliminate $E$. Apply this deformation to all the chart loops on 1-handles, and we have 
\begin{equation*}
(F, \Gamma_0) +\sum_{j=1}^{N-1}h_j(e, e)+\sum_\lambda 1(e,e). 
\end{equation*}
Since we first attached trivial 1-handles, by a deformation which recovers the original 1-handles, we can deform $h_1, \ldots, h_{N-1}$ to  trivial 1-handles; hence we have the required result. 
\end{proof}

\section{Example}\label{sec8}

We show an example. 

\begin{figure}[ht]
\centering
\includegraphics*[height=5cm]{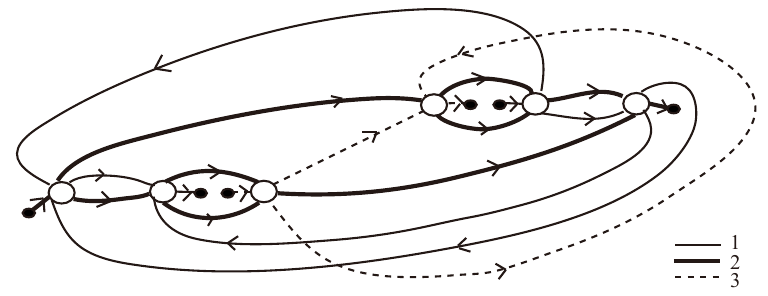}
\caption{The chart $\Gamma$, where we regard $\Gamma$ as drawn on $S^2$.}
\label{fig21}
\end{figure}

\begin{proposition}
Let $(S^2, \Gamma)$ be a 2-dimensional braid where $S^2$ is the 2-sphere standardly embedded in $\mathbb{R}^4$ and $\Gamma$ is the chart illustrated in Figure \ref{fig21}. As a surface-knot, $(S^2, \Gamma)$ presents a 2-twist-spun trefoil \cite[Section 21.4]{Kamada02}. 
Then, the unbraiding number and the weak unbraiding number of $(S^2, \Gamma)$ is one:

\[
u(S^2, \Gamma)=u_w(S^2, \Gamma)=1.
\]

\end{proposition}

\begin{figure}[ht]
\centering
\includegraphics*[width=13cm]{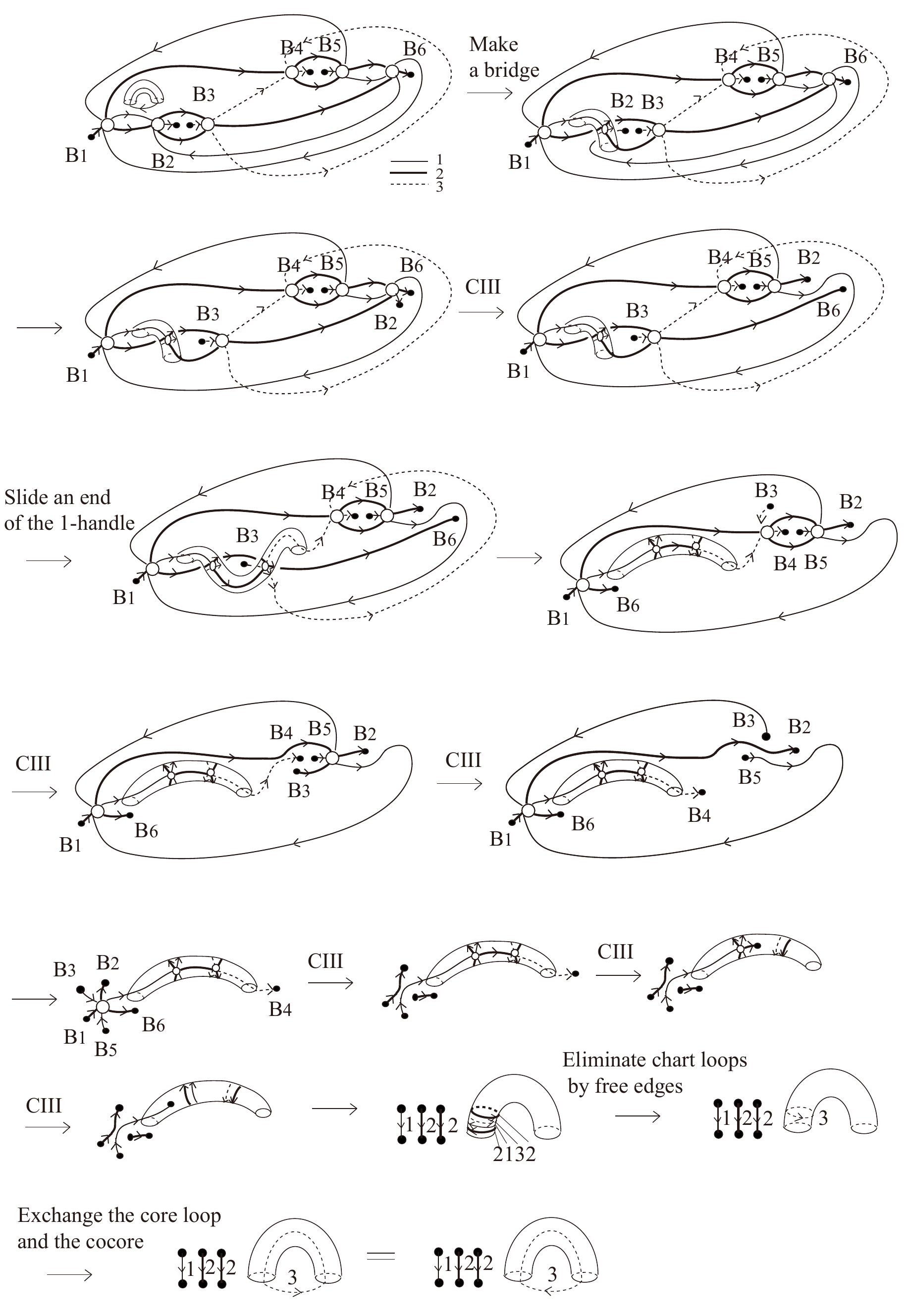}
\caption{Unbraiding $\Gamma$ by an addition of a 1-handle $1(\sigma_1, e)$.}
\label{fig22}
\end{figure}

We remark that it is known \cite[Section 31.3]{Kamada02} that $\Gamma$ is deformed to an unknotted chart by an addition of a free edge, thus $\Gamma$ has the unknotting number one: $u(\Gamma)=1$. 

\begin{proof}
We show that $(S^2, \Gamma)$ can be deformed to the form $(S^2, \Gamma_0)+1(\sigma_3, e)$ for an unknotted chart $\Gamma_0$, by an addition of a 1-handle with a chart loop $1(\sigma_1,e)$. 
Since $\Gamma$ is not equivalent to an unknotted chart, this implies that $u(S^2, \Gamma)=u_w(S^2, \Gamma)=1$. 

We denote by $W_j$ (respectively, $B_j$) the $j$th white vertex (respectively, black vertex) from the left in Figure \ref{fig21} $(j=1,\ldots, 6)$. 
First, add a 1-handle $H=1(\sigma_1, e)$ near $W_2$ as indicated in the first figure in Figure \ref{fig22}, and make a bridge to gather $W_2$ on $H$. Then, $B_2$ is connected with $W_6$. By an ambient isotopy, move $B_2$ near $W_6$. Since $B_2$ and $W_6$ are connected by a non-middle edge, apply a CIII-move to eliminate $W_6$. Then $B_2$ is connected with $W_5$ and $B_6$ is connected with $W_3$. Then, slide an end of $H$ along the diagonal edges of $W_3$ to gather $W_3$ on $H$. Then $B_3$ is connected with $W_4$ and $B_6$ is connected with $W_1$. Apply a CIII-move to eliminate $W_4$. Then $B_4$ is connected with $W_3$ and $B_3$ is connected with $W_5$. 
Apply a CIII-move to eliminate $W_5$. Then $B_2, B_3, B_5$ are connected with $W_1$. Apply a CIII-move to eliminate $W_1$. Then, we have two free edges with the label $2$, and $W_2$ and $W_3$ on $H$ connected with two black vertices. Since the diagonal edges of $W_2$ and $W_3$ connected with the black vertices are non-middle edges, apply CIII-moves twice to eliminate $W_3$ and then $W_2$. Then we have a free edge with the label $1$ and two free edges with the label $2$ and a 1-handle $H=1(e, \sigma_2^{-1}\sigma_1^{-1}\sigma_3\sigma_2)$, which is presented by loops with the labels $2,1,3,2$ along the cocore. By using the free edges, eliminate the loops with the labels $2$ and $1$. Then, $H$ is deformed to $1(e, \sigma_3)$. By exchange of the core loop and the cocore as in Lemma \ref{lem4-4}, $H$ is deformed to $1(\sigma_3^{-1}, e)$, which is equivalent to $1(\sigma_3, e)$. Thus, by an addition of $1(\sigma_1, e)$, $(S^2, \Gamma)$ is deformed to $(S^2, \Gamma_0)+1(\sigma_3, e)$ for an unknotted chart $\Gamma_0$.  
\end{proof}

\section*{Acknowledgements}
The author would like to thank the referee for his/her helpful comments. 
This work was supported by iBMath through the fund for Platform Project for Supporting in Drug Discovery and Life Science Research (Platform for Dynamic Approaches to Living System) from the Ministry of Education, Culture, Sports, Science and Technology, Japan (MEXT) and Japan Agency for Medical Research and Development (AMED), and JSPS KAKENHI Grant Number 15K17532.

\end{document}